\documentclass[11pt]{article}
\usepackage[T1]{fontenc}
\usepackage{authblk}
\usepackage{amsmath,amssymb,amsthm,mathtools}
\usepackage[margin=1in]{geometry}
\usepackage{setspace}
\usepackage{enumitem}
\usepackage{booktabs,longtable,array}
\usepackage{graphicx}
\usepackage{tikz,pgf}
\usepackage{placeins}
\usepackage[numbers,sort&compress]{natbib}
\usepackage{hyperref}
\hypersetup{
  pdfauthor={Tobias Harks},
  pdftitle={Monotonicity, Uniqueness and Frank--Wolfe Dynamics in Atomic Splittable Congestion Games}
}

\newtheorem{theorem}{Theorem}
\newtheorem{lemma}{Lemma}
\newtheorem{proposition}{Proposition}
\newtheorem{corollary}{Corollary}
\newtheorem{definition}{Definition}
\newtheorem{remark}{Remark}

\DeclareMathOperator{\relint}{relint}
\DeclareMathOperator{\lin}{lin}
\DeclareMathOperator{\aff}{aff}
\DeclareMathOperator{\id}{id}
\DeclareMathOperator{\spanop}{span}
\DeclareMathOperator{\diag}{diag}
\DeclareMathOperator{\blkdiag}{blkdiag}
\DeclareMathOperator{\spec}{spec}
\DeclareMathOperator{\Repart}{Re}

\newcommand{\ones}{\mathbf 1}
\newcommand{\N}{\mathbb{N}}	% Natürliche Zahlen
\newcommand{\R}{\mathbb{R}}	% Reelle Zahlen
\newcommand{\Qinf}{Q_\infty}

\newcommand{\load}[1]{\bar #1 }

\newcommand{\G}{\mathcal{G}}
\newcommand{\D}{\mathcal{D}}
\newcommand{\Rp}{\mathbb R_+}

\newcommand{\norm}[1]{\lVert #1 \rVert}
\DeclareMathOperator{\conv}{conv}

\newcommand{\C}{\mathcal{C}}

\newcommand{\ip}[2]{\left\langle #1,#2\right\rangle}

\newcount\colveccount
\newcommand*\colvec[1]{
        \global\colveccount#1
        \begin{pmatrix}
        \colvecnext
}
\def\colvecnext#1{
        #1
        \global\advance\colveccount-1
        \ifnum\colveccount>0
                \\
                \expandafter\colvecnext
        \else
                \end{pmatrix}
        \fi
}
\newcommand*\diff{\,\mathop{}\!\mathrm{d}} % the upright d for integrals
\DeclarePairedDelimiterX{\scalar}[2]{\langle}{\rangle}{#1, #2} % Scalar product
\title{Monotonicity, Uniqueness and Frank--Wolfe Dynamics\\ in Atomic Splittable Congestion Games}

\author{Tobias Harks}
\affil{\small University of Passau\\
\href{mailto:tobias.harks@uni-passau.de}{\texttt{tobias.harks@uni-passau.de}}}
\date{}

\begin{document}
\maketitle

\begin{abstract}
We study atomic splittable congestion games with $n\ge2$ players and
nondecreasing $C^2$ resource costs $f$ with convex $x\mapsto xf(x)$.
We characterize the largest resource cost class for which the Nash
equilibrium is unique. This characterization combines a curvature inequality
involving the first two derivatives and the number of players with strict
increase of the scalar marginal costs $x\mapsto f(x)+xf'(x)$.
The proof rests on first characterizing the largest resource cost classes
for which the associated variational inequality operator is monotone,
strictly monotone, or strongly monotone.
We also draw a perhaps surprising connection to learning dynamics: the same
curvature inequality characterizes universal local and global stability of
Euclidean-regularized Frank--Wolfe dynamics on arbitrary compact convex
strategy spaces. Both the uniqueness and stability characterizations assume
closure under positive affine transformations.
Both characterizations remain exact for network games; the necessity
constructions use acyclic directed graphs whose arc costs can be chosen
strictly positive.
Finally, we show that on parallel-link networks, every interior equilibrium
of the regularized Frank--Wolfe dynamics is locally exponentially stable
even without the curvature condition.
\end{abstract}

\section{Introduction}
Atomic splittable congestion games are a fundamental model in the area of algorithmic game theory
and economics. In this model,  every player has a demand that she may split 
fractionally over allowable subsets of resources. The cost of a resource is then a function
of the total load assigned to it. This class of games has  numerous
applications, e.g.,  in modeling packet-routing in communication networks (see Orda et al.~\cite{Orda93} and Korilis et al.~\cite{Korilis1997,KorilisLO95}), traffic networks (Haurie and Marcotte~\cite{Haurie85}) and logistics networks (Cominetti et al.~\cite{CCS06}).
At the same time, atomic splittable congestion games are a quite intricate class of games
for which even today several key research questions are open.
For example, Bhaskar et al.~\cite{Bhaskar09} constructed a network instance
using very steep convex cost functions for which equilibria are not unique.
Altman et al.~\cite{Altman02competitiverouting} showed that equilibria are
unique for monomial costs plus constants,
$c(x)=a x^d+b$ with $a>0$ and $b\geq0$ and $d\in\{1,2,3\}$. 
Prior uniqueness results for broader cost classes required special assumptions
on the strategy spaces (Bhaskar et al.~\cite{Bhaskar09,Bhaskar15}, Harks and Timmermans~\cite{HarksT18} and Richman and Shimkin~\cite{Richman07}). 
From the perspective of equilibrium computation, the landscape
is also fragmented.   For rational \emph{affine} player-independent cost functions on explicitly represented flow polytopes, there
exists a convex potential whose global minima are pure Nash equilibria (Cominetti et al.~\cite{CCS06}) and, thus, convex programming methods
can be used to compute $\epsilon$-approximate
equilibria in polynomial time. For affine player-specific costs, Klimm and
Warode~\cite{KlimmWarode2025} established PPAD-hardness results. For nonlinear
polynomial costs, the general computational picture remains incomplete:
Bhaskar and Lolakapuri~\cite{BhaskarL18} derive NP-hardness of equilibrium
computation with additional requirements such as a cost bound, while positive
results are known for restricted strategy spaces~\cite{Huang11,BhaskarL18,HarksT22}.

Let us now turn to the area of \emph{learning algorithms} for congestion games. 
For nonatomic congestion games with nondecreasing separable costs, Wardrop equilibria are both minimizers of the Beckmann potential~\cite{beckmann1956} and solutions of a monotone variational inequality (VI)~\cite{facchinei2003}. Potential descent and monotonicity underlie convergence of many learning dynamics. Hofbauer and Sandholm~\cite{hofbauer09} and Sorin and Wan~\cite{sorin2016finite} study such dynamics under monotonicity (called stability and dissipativity, respectively), the latter also for atomic splittable congestion games. Mertzios~\cite{mertzios2009fast} obtained convergence results for two-player atomic splittable games on series-parallel networks.
Hadikhanloo et al.~\cite{Hadikhanloo2021} provide an elegant framework unifying several of these results. Chen et al.~\cite{ChenKA24}
recently derived a general convergence result for best-response VI dynamics (in the spirit of Frank--Wolfe dynamics~\cite{Frank56}) under the monotone VI condition. 

For unweighted atomic unsplittable congestion games, Rosenthal's exact potential~\cite{rosenthal1973} implies the finite improvement property and convergence of asynchronous better-response dynamics to a pure Nash equilibrium~\cite{Awerbuch08,monderer1996,fabrikant2004}.

\subsection{Our Results and Techniques}
Our two main results characterize the largest cost classes that guarantee uniqueness of
Nash equilibria and stability of the regularized FW dynamics, respectively. To derive these characterizations, we first consider universal monotonicity of the underlying VI operator. This property already gives universal stability of the regularized FW dynamics, and the necessity constructions will also be useful for the characterizations of uniqueness and stability.
\paragraph{VI Monotonicity.}
As potential functions are unavailable in general for atomic splittable congestion games
already with quadratic costs, a more promising way to obtain convergence
results for learning dynamics is through VI monotonicity.  
The first topic of this paper is to understand under which structural conditions on the resource cost functions, the resulting game is 
monotone and allows the application of known convergence results. 

Let $\D(\R_+)$ be the set of $C^2$ cost functions on
$\R_+$ satisfying
\[
 f'(x)\ge0
 \qquad\text{and}\qquad
 \Theta_f(x):=2f'(x)+x f''(x)\ge0
 \quad\forall x\in\R_+.
\]
Equivalently, $f$ is nondecreasing and the individual resource cost
$x\mapsto x f(x)$ is convex.  This is the natural convexity assumption in
atomic splittable congestion games, because it guarantees convexity of each
player's optimization problem in her own resource usage.  We say that a class
of cost functions $\C\subset\D(\R_+)$ is $n$-player VI-monotone if every atomic
splittable congestion game with $n$ players and cost functions in $\C$ has a
monotone marginal-cost VI operator.

Our first question is: what is the largest class of $n$-player VI-monotone cost
functions?  To state the answer, fix $n\ge 2$ and set
\[
\gamma_n:=2+\frac{4}{n-1}=\frac{2(n+1)}{n-1}.
\]
For $f\in\D(\Rp), x\in \R_+$, define the upper curvature margin
$\Delta_f(x):=\gamma_n f'(x)-xf''(x).$
We use the following cost-function classes:
\begin{align*}
Q_n(\Rp)
&:=
\left\{
f\in\D(\Rp):
\Delta_f(x)\ge0\ \forall x\in\Rp
\right\},\\
Q_n^{\mathrm{sm}}(\mathbb R_+)&:=
\left\{
f\in \mathcal D(\mathbb R_+):
\Delta_f(s)>0
\text{ and }
\Theta_f(s)>0
\text{ for all }s\in\mathbb R_+
\right\}\\
\D_{\rm mc-si}(\Rp)
&:=
\left\{
f\in\D(\Rp):
\begin{array}{l}
x\mapsto f(x)+x f'(x)\text{ is strictly increasing}\\[-1mm]
\text{on every nondegenerate interval\footnotemark}
\end{array}
\right\}.
\end{align*}
\footnotetext{An interval is nondegenerate if it contains at least two distinct
points; thus $[a,b]$ is nondegenerate precisely when $a<b$.}
The classes satisfy
\[
Q_n^{\mathrm{sm}}(\mathbb R_+)
\subseteq
Q_n(\Rp)\cap\D_{\rm mc-si}(\Rp)
\subseteq
Q_n(\Rp).
\]
Our first result gives a complete characterization of the corresponding
universal VI-monotonicity notions, which are formally defined in
Section~\ref{sec:vi-operator}.

\begin{theorem}[Universal monotonicity classes]\label{thm:classification}
Let $\C\subseteq\D(\Rp)$. Then:
\begin{enumerate}
\item\label{enum:char1} $\C$ is $n$-player VI-monotone if and only if
$\C\subseteq Q_n(\Rp)$.
\item\label{enum:char2} $\C$ is $n$-player strictly VI-monotone if and
only if $\C\subseteq Q_n(\Rp)\cap \D_{\rm mc-si}(\Rp)$.
\item\label{enum:char3} $\C$ is $n$-player strongly VI-monotone if and only if
$\C\subseteq Q_n^{\mathrm{sm}}(\mathbb R_+)$.
\end{enumerate}
\end{theorem}

The proof analyzes the marginal-cost Jacobian resource by resource.
Lemma~\ref{ineq:main} gives sharp bounds on its potentially indefinite term,
yielding the upper curvature threshold $\Delta_f\ge0$, while $\Theta_f\ge0$
controls negative curvature. For necessity, any $f\notin Q_n(\Rp)$ yields
an $n$-player game with one resource and interval strategy sets whose
marginal-cost operator is not monotone.

\paragraph{Frank--Wolfe Stability.}
Theorem~\ref{thm:classification} has immediate implications for learning dynamics.
Consider, for instance, the best-response inclusion
$\dot x\in \beta(F(x))-x$,
where $x$ is a strategy profile, $F$ is the marginal cost operator and 
$\beta(z)
:=
\arg\min_{y\in K}
\left\{
\langle z,y\rangle
\right\}$
and $K$ denotes the convex and compact strategy space of the game.
Note that the set-valued map $\beta$ is exactly the linear minimization oracle used in the Frank--Wolfe (FW), or
conditional-gradient, method. It minimizes a linearization using current marginal costs,
rather than the atomic player's full cost against fixed opponents.
Using a recent result of Chen et al.~\cite{ChenKA24}(Thm.1), we obtain
the following result.
\begin{corollary}\label{cor:FW}
For any atomic splittable congestion game with $n$ players and cost functions in $Q_n(\Rp)$,
the nonempty equilibrium set is globally asymptotically stable under the FW dynamics.
\end{corollary}

For $\eta>0$, Euclidean regularization replaces the oracle by
\[
 \beta^\eta(z):=\arg\min_{y\in K}
 \left\{\langle z,y\rangle+\frac\eta2\|y\|^2\right\}.
\]
The rest points of $\dot x=\beta^\eta(F(x))-x$ solve
$\operatorname{VI}(K,F+\eta\operatorname{id})$ and are called
\emph{regularized equilibria}. For fixed $\eta>0$, they need not be Nash
equilibria of the original game.
For games with costs in $Q_n(\Rp)$, Proposition~\ref{prop:RFW} gives global
asymptotic stability of the unique regularized equilibrium; its proof is
given in Section~\ref{apx:global-conv}.

The second main contribution shows that the curvature condition of $Q_n(\Rp)$ is not merely a
sufficient condition for global convergence of
the Euclidean-regularized FW dynamics but, under a
natural closure property, also \emph{necessary} for universal global and local stability, respectively.  For a cost function $f$, let
\[
\mathrm{Aff}_+(f):=\{\alpha f+\beta:\alpha>0,\ \beta\in\R\}
\]
be its positive affine hull.  We call a class $\C\subseteq \D(\Rp)$ \emph{PAT-closed},  if
it is closed under \emph{positive affine transformations}, that is,
$\mathrm{Aff}_+(f)\subseteq\C$ for every $f\in\C$.
We call a class $\mathcal C$ \emph{universally locally FW-stable} if, for
 every $\eta>0$ and every $n$-player atomic splittable congestion game with
 resource costs in $\mathcal C$, every regularized equilibrium of the
 Euclidean-regularized FW dynamics is locally Lyapunov stable. We call it
 \emph{universally globally FW-stable} if the regularized-equilibrium set is
 globally asymptotically stable in every such game and for every $\eta>0$.

\begin{theorem}[FW-stability characterization]\label{thm:fw-characterization}
Let $n\ge2$ and let $\mathcal C\subseteq\mathcal D(\mathbb R_+)$ be
PAT-closed. Then the following are equivalent:
\begin{enumerate}
\item $\mathcal C$ is universally locally FW-stable.\label{enum:uls}
\item $\mathcal C$ is universally globally FW-stable.\label{enum:ugs}
\item $\mathcal C\subseteq Q_n(\mathbb R_+)$.\label{enum:mon}
\end{enumerate}
Consequently, $Q_n(\mathbb R_+)$ is the unique largest PAT-closed class
with either universal FW-stability property.
\end{theorem}
Sufficiency follows from Proposition~\ref{prop:RFW}. For necessity, a
curvature violation yields a game with interval strategy sets and costs in
$\mathrm{Aff}_+(f)$ having an isolated, linearly unstable interior
regularized equilibrium, which excludes both universal stability properties.

Our main construction is illustrated
for a three-player game using degree-six polynomial cost functions.
\begin{corollary}[Degree-$6$ polynomial instability]\label{cor:degree-six-counterexample}
There is an atomic splittable congestion game with three players, polynomial
cost functions with nonnegative coefficients and degree bounded by $6$, and an
interior regularized equilibrium that is locally unstable under the Euclidean-regularized
Frank--Wolfe dynamics.
\end{corollary}
Note that the degree bound is tight for this instability result.  When $n=3$, the
curvature condition is $x f''(x)\le4f'(x)$, and every polynomial with
nonnegative coefficients and degree at most five satisfies this inequality term
by term.  Thus degree six is the first degree at which the critical curvature
threshold can be violated.

\paragraph{Uniqueness of Nash Equilibria.}
The construction used in the proof of
Theorem~\ref{thm:fw-characterization} can also be used to obtain a matching
necessity result for universal uniqueness of Nash equilibria. Informally, a
cost class has this property if every $n$-player atomic splittable congestion
game with resource costs from that class, irrespective of the compact convex
strategy sets, has exactly one Nash equilibrium. Thus the same cost class that
characterizes universal strict monotonicity exactly characterizes universal
uniqueness of Nash equilibria.

\begin{theorem}[Uniqueness characterization]
\label{thm:uniqueness-characterization}
Let \(n\ge2\), and let \(\C\subseteq\D(\Rp)\) be PAT-closed. Then every
\(n\)-player atomic splittable congestion game with compact convex strategy
sets and resource costs in \(\C\) has a unique Nash equilibrium if and only if
$ \C\subseteq Q_n(\Rp)\cap\D_{\rm mc-si}(\Rp).$
\end{theorem}
For an overview of the structure of $Q_n(\Rp)$, see
Section~\ref{sec:cost-classes}.

\paragraph{Network Representations.}
Our fourth result shows that both the FW stability and uniqueness
characterizations remain exact for network games. The counterexamples used for
necessity can even be chosen on acyclic directed graphs.

\begin{theorem}[FW stability and uniqueness in network games]
\label{thm:directed-network-characterizations}
Let $n\ge2$ and let $\C\subseteq\D(\Rp)$ be PAT-closed. The following
are equivalent:
\begin{enumerate}
\item\label{enum:network-fw-local} Every regularized equilibrium is locally
Lyapunov stable under the Euclidean-regularized FW dynamics in every
$n$-player directed network game with costs in $\C$, for every $\eta>0$.
\item\label{enum:network-fw-global} The regularized-equilibrium set is globally
asymptotically stable in every such game, for every $\eta>0$.
\item\label{enum:network-fw-class} $\C\subseteq Q_n(\Rp)$.
\end{enumerate}
Moreover, every $n$-player directed network game with costs in $\C$ has a
unique Nash equilibrium if and only if
$\C\subseteq Q_n(\Rp)\cap\D_{\rm mc-si}(\Rp)$.
The counterexamples for necessity can be chosen acyclic with strictly
positive arc costs.
\end{theorem}

The FW converse uses a two-layer network that realizes two copies of the
extremal resource direction from the necessity proof of
Theorem~\ref{thm:classification}. For uniqueness, we join the two layers
diagonally, which removes the transportation-kernel directions, and extend the
Brouwer-degree argument to a strictly complementary equilibrium on a face of
the flow polytope. Vertex-potential shifts preserve equilibria and the FW
dynamics while making all arc costs strictly positive
(Lemma~\ref{lem:positive-network-costs}). In particular,
Corollary~\ref{cor:directed-degree-six} gives degree-six counterexamples
with nonnegative coefficients on acyclic directed networks.

Our fifth and final result identifies a contrasting phenomenon on parallel-link
networks, whose fixed-demand strategy spaces are simplices. Regularized
Frank--Wolfe dynamics are locally stable there under weaker conditions than
VI-monotonicity.

\begin{theorem}[Local stability on parallel-link networks]\label{thm:parallel}
Let $\G$ be an atomic splittable congestion game with $n\ge2$ players and
costs in $\D(\Rp)$, where each player sends demand $d_i>0$ over the same
set of parallel links:
\[
 K_i=\left\{x_i\in\Rp^E:\sum_{e\in E}x_{i,e}=d_i\right\},
 \qquad i\in N.
\]
For every $\eta>0$, every interior rest point $x^*\in\relint(K)$ of the
Euclidean-regularized FW dynamics is locally exponentially stable.
\end{theorem}

\subsection{Significance of Our Results}
Atomic splittable congestion games have applications in 
scheduling games~\cite{Korilis1997,KorilisLO95}, routing games~\cite{Haurie85,RS11a,Orda93}, facility location games~\cite{krogmann2024}, network design~\cite{FOTAKIS2014} and logistics networks~\cite{CCS06}.
In all of the above applications, myopic distributed play of the players should guide the system to a stable state. Because the players' feasible strategy sets may be privately known and unavailable to the system designer, it is natural to study stability conditions expressed solely in terms of the cost functions and the number of players. In fact, in most of the above-mentioned applications, the cost functions
are under control of the system designer since they represent the technology associated with the resources,
e.g., queuing discipline at routers, latency function in transportation networks, etc.
Therefore, our characterization results precisely explain under which cost structure 
the system is stable under the (regularized) VI best-response dynamics, \emph{regardless} of the individual strategy spaces. 

Stability alone  does not fully predict the outcome of a game: if
there are multiple equilibria it is not clear upfront which equilibrium will be selected
by the players. This issue has been raised explicitly by Aumann~\cite{Aumann1985} quoted below:
“…it is by no means clear how the players would arrive at an equilibrium, why they
should play equilibrium strategies, and how a specific equilibrium would be chosen
from among the set of all equilibria”. Our characterization
identifies when the resource-cost technology alone guarantees  a determinate
prediction, regardless of the players' privately known feasible sets. 
Finally, our FW-stability and uniqueness characterizations remain exact for~\emph{network games}, where the strategy spaces are given by player-specific flow polytopes on a directed graph.
\section{Preliminaries}
In the following we consider $\R^d$ with its standard inner
product denoted by $\langle \cdot, \cdot\rangle$ with induced norm $\norm{\cdot}$.
We denote by $\R^d_+$ the (component-wise) non-negative real vectors in $\R^d$.
For a finite set \(E\), we write \(\mathbb R^E\) for the Euclidean space
of vectors indexed by \(E\). Thus a vector \(x_i\in\mathbb R_+^E\) is
written as \(x_i=(x_{i,e})_{e\in E}\).
\subsection{Atomic Splittable Congestion Games}
\label{sec:ascg}
An atomic splittable congestion game can be represented by a tuple 
$\G=(N,E,(K_i)_{i\in N}, (c_e)_{e\in E})$, where  $E=\{e_1,\dots,e_m\}$ is a finite set of resources
and $N=\{1,\dots,n\}$ with $n\in \N$ is the finite player set. 
In a quite general formulation, each player $i\in N$ is associated with a player-specific nonempty, compact, convex
set $K_i\subset\R_+^E$ representing the set of feasible non-negative strategy vectors.
Note that $K_i$ may for instance contain the set of edge load vectors representing
$s_i$-$t_i$ flows of demand $d_i>0$ in a given directed or undirected graph, where $s_i$ and $t_i$
are the source and sink nodes, respectively. More generally, $K_i$ may represent the
resource load vectors when distributing the demand over sets of a player-specific set system over the resources (general congestion game formulation).

The overall strategy space is given by $K := \bigtimes _{i \in N} K_i$ and ${x}=({x}_i)_{i\in N}$ 
with $x_i=(x_{i,e})_{e\in E}\in K_i$ is called strategy profile.
For a strategy profile $x$ and fixed resource $e\in E$ we denote by
$x_e:=(x_{i,e})_{i\in N}$ the vector of player-specific usages of $e$.
The \emph{load} induced by a strategy profile $x\in K$ on resource $e\in E$ is
$\load{x}_{e}:=\sum_{i\in N} x_{i,e}$ for every $e\in E$.
Resources are associated with nondecreasing and twice continuously differentiable load-dependent cost functions $c=(c_e)_{e\in E}$ with
$c_e:\R_+\rightarrow \R,e\in E$. We assume $c_e\in \D(\R_+)$ for all $e\in E$, where
\[
\D(\R_+):=\left\{f\in C^2(\R_+)\mid
f'(x)\geq0,\; 2f'(x)+x f''(x)\geq 0
\text{ for all }x\in\R_+\right\}.
\]
Equivalently, $f\in\D(\R_+)$ if $f$ is nondecreasing and $x\mapsto x f(x)$ is convex.  This is exactly the convexity condition needed for each player's private cost in her own resource usage: if $s$ is the total load and $0\le r\le s$ is the player's own load on a resource, then
$2c'_e(s)+r c''_e(s)\ge0$,
because either $c''_e(s)\ge 0$ or $2c'_e(s)+s c''_e(s)\ge 0$.
The private cost function of player $i\in N$ is
$\pi_i:K\rightarrow\R$, where
$\pi_i(x)=\langle c(x),x_i\rangle=\sum_{e\in E}c_e(\load{x}_{e})x_{i,e}$.
Resource costs may take negative values, as allowed by the positive affine
transformations used in our necessity constructions. On acyclic networks,
vertex-potential shifts preserve Nash equilibria and the regularized FW
dynamics while making all arc costs positive; see
Lemma~\ref{lem:positive-network-costs}.
We
write $K_{-i} = \bigtimes_{j \neq i} K_j$ and we write $x = (x_i,x_{-i})$ for each $i \in N$,
meaning that $x_i \in K_i$ and $x_{-i} \in K_{-i}$. 
\begin{definition}
A strategy profile
$x\in K$ is a \emph{Nash equilibrium} for a game $\G$, if $\pi_i(x) \leq
\pi_i(y_i, x_{-i})$ for all $i \in N$ and $y_i \in K_i$. 
\end{definition}
Using compactness and convexity of the strategy spaces together with convexity of each player's cost in her own strategy, Kakutani's fixed point theorem~\cite{Kakutani41} implies the existence of at least one Nash equilibrium.

\subsection{Variational Inequalities}
\label{sec:vi-operator}
We define the \emph{marginal cost operator}
$F:\R_+^{N\times E}\to\R^{N\times E}$ by
\[
 F_{i,e}(x)
 :=
 c_e(\load{x}_{e})+x_{i,e}c'_e(\load{x}_{e}),
 \qquad i\in N,\ e\in E.
\]
Equivalently, for each resource $e\in E$ we define the resource block
\[
 F_e:\R_+^N\to\R^N,
 \qquad
 F_e(z)=\bigl(c_e(s)+z_i c'_e(s)\bigr)_{i\in N},
 \qquad
 s=\ones^\top z,
\]
so that $F_{i,e}(x)=(F_e(x_e))_i$. Here, $\ones$ denotes the all ones
vector.
We obtain the following \emph{variational inequality} that provides sufficient and necessary equilibrium conditions for a Nash equilibrium:
\begin{lemma}[cf. Harks~\cite{Harks:stack2011}(Lemma 1)]
	\label{equilibriumcondition}
	A strategy profile $x\in K$ is a Nash equilibrium if and only if  $x\in K$ solves the following variational inequality termed $VI(K,F)$: 
	\[ \langle F(x),x-y\rangle=\sum_{i\in N}\sum_{e\in E}F_{i,e}(x)(x_{i,e}-y_{i,e})\leq 0 \text{ for all }y\in K.\]
\end{lemma}

We next state the three cost-class VI notions used in
Theorem~\ref{thm:classification}.  These definitions are universal in the sense
that they require the stated property for every atomic splittable congestion game
with the prescribed number of players and with all resource costs drawn from the
class.

\begin{definition}[Universal VI-monotonicity]
\label{def:universal-vi-monotonicity}
A class \(\C\subseteq\D(\R_+)\) is called \(n\)-player VI-monotone if every
atomic splittable congestion game with \(n\) players and all resource costs in
\(\C\) has a monotone VI operator:
\[
\ip{F(x)-F(y)}{x-y}\ge0
\qquad \forall x, y \in K.
\]
\end{definition}

\begin{definition}[Universal strict VI-monotonicity]
\label{def:universal-strict-vi-monotonicity}
A class \(\C\subseteq\D(\R_+)\) is called \(n\)-player strictly VI-monotone if
every such game has a strictly monotone VI operator:
\[
x, y \in K,\ x\ne y
\quad\Longrightarrow\quad
\ip{F(x)-F(y)}{x-y}>0.
\]
\end{definition}

\begin{definition}[Universal strong VI-monotonicity]
\label{def:universal-uniform-strong-vi-monotonicity}
A class \(\C\subseteq\D(\R_+)\) is called \(n\)-player  strongly
VI-monotone if every such game admits a constant
\(\mu>0\) such that
\[
\ip{F(x)-F(y)}{x-y}\ge \mu\norm{x-y}^2
\qquad
\forall x, y \in K.
\]
\end{definition}

\subsection{The Jacobian \texorpdfstring{$DF$}{DF}}
\label{sec:jacobian}
As cost functions are assumed to be $C^2$, we can speak
about the Jacobian of $F$.  Throughout the paper, $DF(x)$ denotes the
ordinary Jacobian of the global operator $F$, while $DF_e(z)$ denotes the
Jacobian of the single-resource block $F_e$ at the resource-usage vector $z$.
Since $F_{i,e}$ depends only on the vector $x_e=(x_{j,e})_{j\in N}$, the global
Jacobian decomposes resource-wise: after grouping coordinates by resources,
\[
 DF(x)=\blkdiag\bigl(DF_e(x_e)\bigr)_{e\in E},
\]
where $\blkdiag$ stands for the block-diagonal structure.
In the original player-resource ordering this is the same decomposition up to a
simultaneous permutation of rows and columns.
Let $z=(z_i)_{i\in N}\in\R_+^N$, set $s=\ones^\top z$ and abbreviate
$a_e(z):=c'_e(s)$ and $b_e(z):=c''_e(s)$.  A direct computation gives (using $\delta_{ij}$ as the Kronecker-Delta)
\[
 (DF_e(z))_{ij}
 =
 a_e(z)+\delta_{ij}a_e(z)+b_e(z)z_i,
 \qquad i,j\in N.
\]
Equivalently, in matrix form,
\[
 DF_e(z)
 =
 a_e(z)I_n+a_e(z)\ones\ones^\top+b_e(z)z\ones^\top .
\]
Thus, for every $h\in\R^{N\times E}$ with
$\bar h_e:=\sum_{j\in N}h_{j,e}$, we get
\begin{equation}\label{eq:DF-action-prelim}
 (DF(x)h)_{i,e}
 =
 c'_e(\load{x}_{e})h_{i,e}
 +\bigl(c'_e(\load{x}_{e})+x_{i,e}c''_e(\load{x}_{e})\bigr)\bar h_e.
\end{equation}
For a product strategy space $K=\bigtimes_{i\in N}K_i$, we write  $T_K:=\lin(K-K)$
for the tangent space of the affine hull of $K$.  If a game is fixed, we often
write simply $T$ instead of $T_K$.  The orthogonal projection onto this tangent
space is denoted by $P_T$.  For a differentiable operator $H$, we use the
notation
\[
 DH(x)|_T:=P_TDH(x)P_T:T\to T
\]
for the Jacobian restricted to the tangent space.  Thus, for $h\in T$,
$DH(x)|_T h=P_TDH(x)h$.

\section{Characterizing Universally Monotone Cost Functions}
\label{sec:characterization}

This section is devoted to the proof of Theorem~\ref{thm:classification}.
Throughout this section the number of players $n\ge 2$ is fixed, and we use the
notation from Section~\ref{sec:ascg}.   Recall that
\[
\gamma_n:=2+\frac{4}{n-1}=\frac{2(n+1)}{n-1}
\]
and, for $f\in\D(\Rp)$,
\[
\Delta_f(s):=\gamma_n f'(s)-s f''(s),
\qquad
\Theta_f(s):=2f'(s)+s f''(s),
\qquad s\in\Rp.
\]
The condition $f\in\D(\Rp)$ is equivalently $f'\ge 0$ and $\Theta_f\ge0$, while
$f\in Q_n(\Rp)$ is equivalently $\Delta_f(s)\ge 0$ for all $s\in\Rp$.

The proof uses the resource-block notation introduced
in Section~\ref{sec:vi-operator}.  Fix a resource $e$, suppress the index $e$, write
$f=c_e$, and let $z=(z_i)_{i\in N}\in\Rp^n$ be the vector of player usages of
this resource.  Set
\[s=\ones^\top z,\qquad
 a=f'(s),\qquad b=f''(s).
\]
The Jacobian block of the VI operator on this resource is
\begin{equation}\label{eq:resource-jacobian-block}
DF_e(z)=aI+a\ones\ones^\top+bz\ones^\top .
\end{equation}
Consequently, for every direction $u\in\R^n$,
\begin{equation}\label{eq:block-quadratic-form}
 u^\top DF_e(z)u
 =
 a\norm{u}^2+a(\ones^\top u)^2
 +b(z^\top u)(\ones^\top u).
\end{equation}
Although $DF_e(z)$ is generally not symmetric, only the quadratic form in
\eqref{eq:block-quadratic-form} is relevant for monotonicity.
The following inequality drives our main monotonicity characterization. 
\begin{lemma}[Key inequality]\label{ineq:main}
For every $z\in\Rp^n$ and $u\in\R^n$, with $s=\ones^\top z$, we have
\begin{equation}\label{eq:antagonistic-cauchy}
 (z^\top u)(\ones^\top u)
 \ge
 -\frac{n-1}{2(n+1)}\,s\left(\norm{u}^2+(\ones^\top u)^2\right)
 =
 -\frac{s}{\gamma_n}\left(\norm{u}^2+(\ones^\top u)^2\right),
\end{equation}
and
\begin{equation}\label{eq:cooperative-cauchy}
 (z^\top u)(\ones^\top u)
 \le
 \frac{s}{2}\left(\norm{u}^2+(\ones^\top u)^2\right).
\end{equation}
The factor $(n-1)/(2(n+1))=1/\gamma_n$ in the lower estimate is optimal.
Equality in \eqref{eq:antagonistic-cauchy} is attained, for example, by
$z=se_1$ and $u_1=-1, u_2=\cdots=u_n=\frac{2}{n-1}.$
\end{lemma}

\begin{proof}
We first prove the lower estimate~\eqref{eq:antagonistic-cauchy}.
If $s=0$ or $\ones^\top u=0$, the lower estimate is immediate (note that
$s=0$ implies $z=0$).  By replacing $u$ by $-u$ if necessary, assume
$\sigma:=\ones^\top u>0$.  Since $z/s$ is a convex combination of the unit
vectors, we get the estimate
$z^\top u\ge s\min_i u_i$.
Since both sides of~\eqref{eq:antagonistic-cauchy} are homogeneous of degree two
in $u$, we may rescale $u$. Thus, after assuming
$\sigma:=\mathbf 1^\top u>0$, we replace $u$ by $u/\sigma$ and may
suppose without loss of generality that $\mathbf 1^\top u=1$.
Set \(k:=\min_i u_i\). If \(k\ge 0\), then
\(z^\top u\ge sk\ge0\), and~\eqref{eq:antagonistic-cauchy} follows immediately.
Hence we may assume \(k<0\).
Among all vectors $u\in\mathbb R^n$ satisfying
$\mathbf 1^\top u=1$ and $\min_i u_i=k$, the squared Euclidean norm
$\|u\|^2=\sum_i u_i^2$ is minimized when one coordinate equals $k$
and the remaining $n-1$ coordinates are all equal to $(1-k)/(n-1)$.
Hence,
\[
\|u\|^2
\ge
k^2+\frac{(1-k)^2}{n-1}.
\]
With $r=-k\ge 0$ this gives
\[
 \frac{n-1}{2(n+1)}\left(\norm{u}^2+1\right)+k
 \ge
 \frac{n-1}{2(n+1)}
 \left(r^2+\frac{(1+r)^2}{n-1}+1\right)-r
 =
 \frac{n(r-1)^2}{2(n+1)}\ge0.
\]
Therefore
\[
k\ge -\frac{n-1}{2(n+1)}\bigl(\|u\|^2+1\bigr).
\]
With
$z^\top u\ge s\min_i u_i=sk$ and  \(\mathbf 1^\top u=1\), this yields
\[
(z^\top u)(\mathbf 1^\top u)
=z^\top u
\ge sk
\ge
-\frac{n-1}{2(n+1)}s
\bigl(\|u\|^2+(\mathbf 1^\top u)^2\bigr).
\]
For the upper estimate~\eqref{eq:cooperative-cauchy}, let $\sigma:=\ones^\top u$.  For every coordinate,
we get the inequality 
\[ u_i\sigma\le (u_i^2+\sigma^2)/2\le(\norm{u}^2+\sigma^2)/2.\]  
Multiplying by $z_i\ge0$ and summing over $i$ yields
\[
(z^\top u)(\mathbf 1^\top u)
=
\sum_i z_i u_i\sigma
\le
\frac12\bigl(\|u\|^2+\sigma^2\bigr)\sum_i z_i
=
\frac{s}{2}
\bigl(\|u\|^2+(\mathbf 1^\top u)^2\bigr),
\]
showing
\eqref{eq:cooperative-cauchy}.
\end{proof}
Moreover, the equality case in~\eqref{eq:antagonistic-cauchy} is characterized as
follows. If \(s>0\) and \(\sigma:=\mathbf 1^\top u\ne0\), then equality in~\eqref{eq:antagonistic-cauchy} can hold only if, after a permutation of coordinates,
\[
z=se_k,\qquad
u_k=-\sigma,\qquad
u_j=\frac{2\sigma}{n-1}\quad(j\ne k).
\]
This follows because equality must hold both in \(z^\top u\ge s\min_i u_i\) and in the
norm-minimization step; the latter forces one coordinate to equal
\(-\sigma\) and all remaining coordinates to equal \(2\sigma/(n-1)\), while
the former forces \(z\) to be supported on the minimizing coordinate.

Combining \eqref{eq:block-quadratic-form} with Lemma~\ref{ineq:main} gives the
central estimate
\begin{equation}\label{eq:block-lower-bound}
 u^\top DF_e(z)u
 \ge
 \rho_f(s)
 \left(\norm{u}^2+(\ones^\top u)^2\right),
 \qquad
 \rho_f(s):=
 \begin{cases}
 \Delta_f(s)/\gamma_n, & f''(s)\ge0,\\[2mm]
 \Theta_f(s)/2, & f''(s)<0.
 \end{cases}
\end{equation}
Indeed, if $f''(s)\ge0$, we use the lower estimate
\eqref{eq:antagonistic-cauchy}; if $f''(s)<0$, we use the upper estimate
\eqref{eq:cooperative-cauchy}.  Thus $f\in\D(\Rp)$ and $\Delta_f\ge0$ imply
that every resource block is monotone.
Conversely, the equality case in Lemma~\ref{ineq:main} shows that the upper
curvature condition is also necessary at the level of a single resource: if
$\Delta_f(s)<0$ at some load $s>0$, then $s f''(s)>\gamma_n f'(s)$ and hence
$f''(s)>0$.  For $z=se_1$ and
$u=(-1,2/(n-1),\ldots,2/(n-1))$ we obtain
\begin{equation}\label{eq:tight-direction}
 u^\top DF_e(z)u=\Delta_f(s)<0.
\end{equation}
Since $\Delta_f$ is continuous, the same strict inequality persists in a
neighbourhood of $z$.

Now we need one more Lemma to deal with functions in $f\in Q_n(\mathbb R_+)\cap \mathcal D_{\mathrm{mc-si}}(\mathbb R_+)$.
\begin{lemma}[Strictness of one resource block]
\label{lem:resource-strictness}
Let $f\in Q_n(\mathbb R_+)\cap \mathcal D_{\mathrm{mc-si}}(\mathbb R_+)$,
and let $F_e$ be the corresponding one-resource marginal cost block. Then
\[
\langle F_e(z^1)-F_e(z^0), z^1-z^0\rangle>0 \text{ for all  $z^0,z^1\in\mathbb R_+^n$ with $z^0\neq z^1$.}
\]
\end{lemma}
\begin{proof}
\label{proof-strictness}
Put $u=z^1-z^0\ne0$, $z(t)=z^0+tu$, $s(t)=\ones^\top z(t)$,
and $\sigma=\ones^\top u$. The resource-block bound gives
\[
 \langle F_e(z^1)-F_e(z^0),u\rangle
 =\int_0^1 u^\top DF_e(z(t))u\,dt\ge0.
\]
Write $\phi(s)=f(s)+sf'(s)$. Since
$(s^2f'(s))'=s\Theta_f(s)\ge0$, a zero of $f'$ at some $s_0>0$
would force $f'=0$ on $[0,s_0]$, contradicting strict increase of
$\phi$. Thus $f'(s)>0$ for $s>0$. If $\sigma=0$, the load is constant
and positive (zero load would force $z^0=z^1=0$), so the integrand is
$f'(s(t))\|u\|^2>0$.

If $\sigma\ne0$, the loads span a nondegenerate interval. Strict increase
of $\phi$ and continuity of $\phi'=\Theta_f\ge0$ give a
$t_0\in(0,1)$ with $\Theta_f(s(t_0))>0$; also $s(t_0)>0$.
If $f''(s(t_0))<0$, \eqref{eq:block-lower-bound} makes the integrand
strictly positive. Otherwise, a zero integrand would force
$\Delta_f(s(t_0))=0$ and equality in
\eqref{eq:antagonistic-cauchy}; the case $f''(s(t_0))=0$ is impossible
because $f'(s(t_0))>0$. The equality characterization therefore gives,
for some $k$,
\[
 z(t_0)=s(t_0)e_k,\qquad
 u_k=-\sigma,\qquad u_j=\frac{2\sigma}{n-1}\quad(j\ne k).
\]
But each affine coordinate $z_j(t)$ is nonnegative on $[0,1]$; if it
vanishes at the interior point $t_0$, its slope is zero. Hence $u_j=0$
for $j\ne k$, a contradiction. The continuous, nonnegative integrand
is therefore positive somewhere, and its integral is strictly positive.
\end{proof}

\begin{proof}[Proof of Theorem~\ref{thm:classification}]
Let $\G=(N,E,(K_i)_{i\in N},(c_e)_{e\in E})$ be an atomic splittable congestion
game with $n$ players, strategy space $K=\bigtimes _{i \in N} K_i$, and VI operator $F$.

We first prove the VI-monotonicity statement~\ref{enum:char1}.  Suppose $c_e\in Q_n(\Rp)$ for
all resources.  For $x,y\in K$, put $h=x-y$ and write
$h_e=(h_{i,e})_{i\in N}$.  Since $K$ is convex,
\[
 F(x)-F(y)
 =
 \int_0^1 DF\bigl(y+t(x-y)\bigr)h\,\diff t.
\]
Using the resource decomposition of $DF$ and \eqref{eq:block-lower-bound}, we get
\begin{equation}\label{eq:global-monotonicity-integral}
 \scalar{F(x)-F(y)}{x-y}
 =
 \int_0^1\sum_{e\in E}
 h_e^\top DF_e\bigl(y_e+t h_e\bigr)h_e\,\diff t
 \ge0.
\end{equation}
This proves sufficiency for universal VI-monotonicity.

For necessity, suppose that $f\in\C$ and $f\notin Q_n(\Rp)$.  Choose
$s>0$ with $\Delta_f(s)<0$, set $z=se_1$, and use the direction $u$ from
\eqref{eq:tight-direction}.  For sufficiently small $\varepsilon>0$ we have
$z+\varepsilon u\in\Rp^n$ and
$u^\top DF_e(z+t\varepsilon u)u<0$ for all $t\in[0,1]$.  Consider the one-resource
game in which player $i$ has the interval
$ K_i:=\conv\{z_i,z_i+\varepsilon u_i\}\subset\Rp$
as strategy space and the single resource has cost $f$.  With
$y=z$ and $x=z+\varepsilon u$,
\[
 \scalar{F(x)-F(y)}{x-y}
 =
 \varepsilon^2\int_0^1
 u^\top DF_e(z+t\varepsilon u)u\,\diff t
 <0.
\]
Thus the VI operator is not monotone.  This proves item~\ref{enum:char1}.

We next prove the strict VI-monotonicity statement~\ref{enum:char2}. 
First assume
$c_e\in Q_n(\mathbb R_+)\cap\mathcal D_{\mathrm{mc-si}}(\mathbb R_+)$
for every resource $e\in E$. Let $x,y\in K$ with $x\ne y$, and put
$h=x-y$. Then there exists at least one resource $e$ with
$x_e\ne y_e$. By Lemma~\ref{lem:resource-strictness},
\[
(F_e(x_e)-F_e(y_e))^\top (x_e-y_e)>0
\]
for every such resource, while the resource contributions are nonnegative
for all other resources by the monotonicity part already proved. Hence
\[
\langle F(x)-F(y),x-y\rangle
=
\sum_{e\in E}
(F_e(x_e)-F_e(y_e))^\top (x_e-y_e)
>0.
\]
Thus the VI operator is strictly monotone.

Conversely, suppose that \(\mathcal C\) is \(n\)-player strictly
VI-monotone. By item~\ref{enum:char1}., necessarily \(\mathcal C\subseteq Q_n(\mathbb R_+)\).
It remains to prove that every \(f\in\mathcal C\) belongs to
\(\mathcal D_{\mathrm{mc-si}}(\mathbb R_+)\). Let
$\phi(s):=f(s)+sf'(s)$.
Since \(f\in\mathcal D(\mathbb R_+)\), we have
$\phi'(s)=2f'(s)+sf''(s)=\Theta_f(s)\ge0$,
so \(\phi\) is nondecreasing. If \(f\notin
\mathcal D_{\mathrm{mc-si}}(\mathbb R_+)\), then there exist \(a<b\) such
that \(\phi(a)=\phi(b)\). Since \(\phi\) is nondecreasing, it is constant
on \([a,b]\).

Consider the one-resource game in which player \(1\) has strategy interval
\(K_1=[a,b]\), while all other players are fixed at zero. For two distinct
profiles with player-$1$ loads \(a\) and \(b\), respectively, the marginal
cost is the same:
$F_1(a)=\phi(a)=\phi(b)=F_1(b).$
Hence,
$\langle F(x)-F(y),x-y\rangle=0$
although \(x\ne y\), contradicting strict monotonicity. Therefore
\(f\in\mathcal D_{\mathrm{mc-si}}(\mathbb R_+)\).

 We now prove the strong VI-monotonicity statement~\ref{enum:char3}. 
 We first prove sufficiency. Let
$c_e\in Q_n^{\mathrm{sm}}(\mathbb R_+)$ for every resource $e$.
Fix a game with compact strategy space $K$. For each resource $e$, let
$L_e:=\max_{x\in K}\bar x_e$.
Since $K$ is compact, $L_e<\infty$. Since $\Delta_{c_e}$ and
$\Theta_{c_e}$ are continuous and strictly positive on $[0,L_e]$, the
number
\[
m_e
:=
\min_{s\in[0,L_e]}
\left\{
\frac{\Delta_{c_e}(s)}{\gamma_n},
\frac{\Theta_{c_e}(s)}{2}
\right\}
\]
is strictly positive. By the resource-block inequality~\eqref{eq:block-lower-bound},
for every resource direction $h_e$ and every feasible load vector $z_e$
with total load at most $L_e$,
\[
h_e^\top DF_e(z_e)h_e
\ge
m_e\|h_e\|^2 .
\]
Let $\mu:=\min_{e\in E}m_e>0$.
Then, for all $x,y\in K$ and $h=x-y$,
\[
\begin{aligned}
\langle F(x)-F(y),x-y\rangle
&=
\int_0^1
\sum_{e\in E}
h_e^\top
DF_e(y_e+t h_e)
h_e\,dt  \\
&\ge
\mu\sum_{e\in E}\|h_e\|^2
=
\mu\|x-y\|^2 .
\end{aligned}
\]
Thus the VI operator is strongly monotone on this fixed compact game.
Conversely, suppose that $f\notin Q_n^{\mathrm{sm}}(\mathbb R_+)$.
If $f\notin Q_n(\mathbb R_+)$, then Theorem~\ref{thm:classification}~(\ref{enum:char1}.) already gives
a game whose VI operator is not monotone, hence not strongly monotone.
It remains to consider the case $f\in Q_n(\mathbb R_+)$, but either
$\Delta_f(s_0)=0$ or $\Theta_f(s_0)=0$ for some $s_0\ge0$.
Note first that the case $s_0=0$ needs no separate treatment: since
$\Delta_f(0)=\gamma_nf'(0)$ and $\Theta_f(0)=2f'(0)$, we have
$\Delta_f(0)=0$ if and only if $\Theta_f(0)=0$, so it is covered by the second
case below.
First suppose that $\Delta_f(s_0)=0$ for some $s_0>0$. Take the
one-resource construction used in the necessity proof of Theorem~\ref{thm:classification}~(\ref{enum:char1}.):
\[
z=s_0 e_1,
\qquad
u=\left(-1,\frac{2}{n-1},\ldots,\frac{2}{n-1}\right).
\]
Then $u^\top DF_e(z)u=\Delta_f(s_0)=0$.
Choose $\varepsilon>0$ small enough so that
$z+t u\in\mathbb R_+^n$ for all $t\in[0,\varepsilon]$, and define
$K_i:=\operatorname{conv}\{z_i,z_i+\varepsilon u_i\}$.
For $0<\tau\le\varepsilon$, set $x=z+\tau u$ and $y=z$. Then
\[
\frac{\langle F(x)-F(y),x-y\rangle}{\|x-y\|^2}
=
\frac{\int_0^1 u^\top DF_e(z+t\tau u)u\,dt
}{\|u\|^2}
\longrightarrow
\frac{u^\top DF_e(z)u}{\|u\|^2}
=0
\]
as $\tau\downarrow0$. Hence no positive strong-monotonicity modulus can
exist for this compact game.
Second suppose that $\Theta_f(s_0)=0$ for some $s_0\ge0$. Consider a
one-resource game with one active player whose strategy interval is
$K_1=[s_0,s_0+\varepsilon]$,
and all other players fixed at zero. Along this interval the scalar
marginal cost is $\phi(s)=f(s)+sf'(s)$, with
$\phi'(s)=\Theta_f(s)$.
For $x=s_0+\tau$ and $y=s_0$,
\[
\frac{(F(x)-F(y))(x-y)}{|x-y|^2}
=
\frac{\phi(s_0+\tau)-\phi(s_0)}{\tau}
\longrightarrow
\phi'(s_0)=\Theta_f(s_0)=0,
\]
as $\tau\downarrow0$. Thus strong monotonicity again fails on this compact game.
Therefore pointwise positivity of both $\Delta_f$ and $\Theta_f$ is
necessary. This proves item~\ref{enum:char3}. and completes the proof.
\end{proof}

\subsection{Examples and closure of monotone cost classes}
\label{sec:cost-classes}
The costs that yield monotonicity for every number of players form
\[
 \Qinf(\Rp):=\bigcap_{n\ge2}Q_n(\Rp)
 =\{f\in\D(\Rp):xf''(x)\le2f'(x)\text{ for all }x\ge0\}.
\]
Thus membership amounts to the three inequalities
$f'\ge0$ and $2f'\pm xf''\ge0$.
\label{sec:closure}
These expressions depend linearly on $f$ and vanish on constants.
Consequently $\Qinf(\Rp)$ is closed under nonnegative finite sums and
positive affine transformations: if $g_r\in\Qinf(\Rp)$, $b_r\ge0$,
$\rho>0$, and $\beta\in\mathbb R$, then
$\rho\sum_r b_rg_r+\beta\in\Qinf(\Rp)$.

Table~\ref{tab:Qinf-examples-condensed} lists concrete families. In each
row, $f=a+bg$ with arbitrary $a\in\mathbb R$ and $b\ge0$ belongs to
$\Qinf(\Rp)$. For the nonconstant elementary families, the last column
gives $R_g(x)=xg''(x)/g'(x)$ for $x>0$.
In the logarithmic-power row, $L=A+\log(1+x)$.
\begin{table}[htbp]
\centering
\caption{Cost families in $\Qinf(\Rp)$ and derivative-ratio certificates.
\label{tab:Qinf-examples-condensed}}
\begingroup
\small
\setlength{\tabcolsep}{4pt}
\renewcommand{\arraystretch}{1.5}
\begin{tabular}{@{}p{0.31\linewidth}p{0.34\linewidth}p{0.28\linewidth}@{}}
\toprule
Family $g(x)$ & Parameter restrictions & $R_g(x)$, $x>0$ \\
\midrule
$a_1x+a_2x^2+a_3x^3$
& $a_1,a_2,a_3\ge0$ & By closure \\
\addlinespace[4pt]
$(x+\tau)^p$
& $\tau>0$, $0<p\le3$ & $(p-1)\dfrac{x}{x+\tau}$ \\
\addlinespace[4pt]
$x^p$
& $p=1$ or $2\le p\le3$ & $p-1$ \\
\addlinespace[4pt]
$\bigl(A+\log(1+x)\bigr)^\alpha$
& $A>0$, $0<\alpha\le3$
& $\dfrac{x}{1+x}\left(\dfrac{\alpha-1}{L}-1\right)$ \\
\addlinespace[4pt]
$\log(1+x^p)$
& $p=1$ or $2\le p\le3$ & $p-1-\dfrac{px^p}{1+x^p}$ \\
\addlinespace[4pt]
$1-(1+\lambda x)^{-p}$
& $\lambda>0$, $0<p\le1$ & $-(p+1)\dfrac{\lambda x}{1+\lambda x}$ \\
\addlinespace[4pt]
$\arctan(\lambda x)$
& $\lambda>0$ & $-\dfrac{2\lambda^2x^2}{1+\lambda^2x^2}$ \\
\addlinespace[4pt]
$\operatorname{arsinh}(\lambda x)$
& $\lambda>0$ & $-\dfrac{\lambda^2x^2}{1+\lambda^2x^2}$ \\
\addlinespace[4pt]
$\sqrt{A+x^2}-\sqrt A$
& $A>0$ & $\dfrac{A}{A+x^2}$ \\
\bottomrule
\end{tabular}
\endgroup

\end{table}

To verify the table, constants ($b=0$) satisfy all three inequalities
directly. Every elementary $g$ listed is $C^2$ on $\Rp$ and has
$g'(x)>0$ for $x>0$; in particular the stated pure-power restrictions
ensure $C^2$ regularity at zero. Direct differentiation gives the ratios
shown. Except for logarithmic powers, their membership in $[-2,2]$
follows immediately from the parameter ranges and
$0\le x/(x+\tau),\ x^p/(1+x^p),\ \lambda x/(1+\lambda x)<1$
for $x>0$ and positive parameters. The inverse-function and square-root
ratios lie respectively in $[-2,0]$, $[-1,0]$, and $[0,1]$.

For logarithmic powers, put $r=\log(1+x)>0$. Then
\[
 R_g(x)=(1-e^{-r})\left(\frac{\alpha-1}{A+r}-1\right).
\]
If $\alpha\le1$, the upper bound is $R_g\le0$; otherwise
$R_g\le2(1-e^{-r})/r\le2$.
If $\alpha\ge1$, the lower bound is $R_g\ge-1$; otherwise
\[
 R_g\ge-(1-e^{-r})\left(1+\frac1r\right)\ge-2,
\]
using $1-e^{-r}\le\min\{1,r\}$. Thus $2g'\pm xg''\ge0$ on
$(0,\infty)$ for every elementary row. These inequalities and
$g'\ge0$ extend to $x=0$ by continuity, including rows with $g'(0)=0$.
The polynomial row follows from closure and the powers $x,x^2,x^3$.

\begin{remark}[Polynomial threshold]
\label{ref:rem-costs}
For every $d>0$ for which $x^d\in C^2(\Rp)$,
$\Delta_{x^d}(x)=d(\gamma_n-d+1)x^{d-1}$ for $x>0$.
Hence $x^d\in Q_n(\Rp)$ exactly when $d-1\le\gamma_n$. Applying the condition term by term shows that all
polynomials of degree at most three with nonnegative nonconstant
coefficients belong to $Q_n(\Rp)$ for every $n\ge2$, and all such
polynomials of degree at most five belong to $Q_3(\Rp)$.
Their constant coefficients may be arbitrary real numbers.
\end{remark}

\section{Characterizing Stability of Frank--Wolfe Dynamics}
\label{sec:instability-general}

We prove Theorem~\ref{thm:fw-characterization} and
Corollary~\ref{cor:degree-six-counterexample}. We first establish global
stability for monotone operators and local spectral criteria, then construct
an isolated unstable equilibrium whenever the curvature bound fails.

\subsection{Euclidean Regularized FW}\label{subsec:erfw}\label{apx:global-conv}
For \(\eta>0\) and Euclidean regularization, define the regularized FW (or VI
best-response) map by
\begin{equation}\label{eq:regularized-best-response}
 \beta_i^\eta(z_i)
 :=
 \arg\min_{y_i\in K_i}
 \left\{
 \scalar{z_i}{y_i}+\frac{\eta}{2}\norm{y_i}^2
 \right\},
 \qquad i\in N,
\end{equation}
where \(z_i\in\R^E\).  
Equivalently, using first-order conditions, \(\beta_i^\eta(z_i)\) is the Euclidean
projection of \(-z_i/\eta\) onto \(K_i\).  We write
\(\beta^\eta(z)=(\beta_i^\eta(z_i))_{i\in N}\) and define
$\Phi^\eta(x):=\beta^\eta(F(x))$.
The associated Euclidean-regularized FW vector field is
\begin{equation}\label{eq:def-G-fw}
 G(x):=\Phi^\eta(x)-x=\beta^\eta(F(x))-x,
\end{equation}
and the Euclidean-regularized FW dynamics are given as
\begin{equation}\label{eq:fw-reg}
 \dot x=G(x)=\beta^\eta(F(x))-x.
\end{equation}
Using first-order conditions, a regularized equilibrium $x^*$ is characterized by the VI
\[  \scalar{F(x^*)+\eta x^*}{x^*-y}\leq 0 \text{ for all }y\in K.\]
\begin{definition}[Stability notions]
\label{def:stability-notions}
\label{sec:apx-stability}
All stability notions are relative to the feasible set $K$.
A rest point $x^*$ of $\dot x=G(x)$ is \emph{Lyapunov stable} if every
neighbourhood $U$ of $x^*$ contains a neighbourhood $V$ such that every
solution starting in $V$ remains in $U$ for all $t\ge0$. It is
\emph{locally asymptotically stable} if it is stable and all solutions
starting in some neighbourhood converge to $x^*$, and
\emph{locally exponentially stable} if, for some neighbourhood $W$ and
constants $M,\alpha>0$,
\[
 \|x(t)-x^*\|\le M e^{-\alpha t}\|x(0)-x^*\|
 \qquad(x(0)\in W,\ t\ge0).
\]
Exponential stability implies asymptotic stability, which implies
Lyapunov stability. An equilibrium set $X^*\subseteq K$ is
\emph{globally asymptotically stable} if it is Lyapunov stable as a set
and globally attractive: for every $\varepsilon>0$ there is a
$\delta>0$ such that
$\operatorname{dist}(x(0),X^*)\le\delta$ implies
$\operatorname{dist}(x(t),X^*)\le\varepsilon$ for all $t\ge0$, and
$\operatorname{dist}(x(t),X^*)\to0$ for every $x(0)\in K$.
Here $\operatorname{dist}(x,X^*)=\inf_{y\in X^*}\|x-y\|$.
\end{definition}

The following proof uses the Fenchel-coupling argument of
Hadikhanloo et al.~\cite[Theorem~4.2]{Hadikhanloo2021}.

\begin{proposition}[Global stability of Euclidean-regularized FW]
\label{prop:RFW-apx}
\label{prop:RFW}
Let $K\subseteq\mathbb R^d$ be nonempty, compact, and convex, and let
$F:K\to\mathbb R^d$ be continuously differentiable and monotone.
For every $\eta>0$, the dynamics
\[
 \dot x=\beta^\eta(F(x))-x,\qquad
 \beta^\eta(z)=\arg\min_{y\in K}
 \left\{\langle z,y\rangle+\frac\eta2\|y\|^2\right\},
\]
have a unique rest point $x^\eta$, which is globally asymptotically stable.
\end{proposition}
\begin{proof}
Rest points solve $\operatorname{VI}(K,F+\eta\operatorname{id})$.
Continuity and compact convexity give existence, and strong monotonicity
with modulus $\eta$ gives uniqueness. The projection formula makes
$\beta^\eta$ $1/\eta$-Lipschitz, so the vector field is Lipschitz on $K$.
Moreover, $x+h(\beta^\eta(F(x))-x)\in K$ for $x\in K$ and $0\le h\le1$.
Thus the field points into the tangent cone of $K$, making $K$ positively
invariant; compactness gives a unique global solution from every point of $K$.

Let $h(y)=\eta\|y\|^2/2+\iota_K(y)$, where $\iota_K$ is the convex
indicator (zero on $K$ and $+\infty$ outside), and let
$h^*(v)=\max_{y\in K}\{\langle v,y\rangle-\eta\|y\|^2/2\}$.
Strong convexity of $h$ makes $h^*$ continuously differentiable, with
$\nabla h^*(-F(x))=\beta^\eta(F(x))$. The Fenchel coupling
\[
 V(x)=h(x)+h^*(-F(x))+\langle F(x),x\rangle
\]
is continuous on $K$, nonnegative, and zero exactly at the unique rest
point. Write $y=\beta^\eta(F(x))$ and $d=y-x=\dot x$.
Oracle optimality gives $\langle F(x)+\eta y,x-y\rangle\ge0$.
Since $x+td\in K$ for $0\le t\le1$, monotonicity also gives
$d^\top DF(x)d\ge0$, including at boundary points. Therefore
\[
 \dot V
 =\langle F(x)+\eta x,d\rangle-d^\top DF(x)d
 \le-\eta\|d\|^2.
\]
LaSalle's invariance principle on the compact invariant set $K$ implies
that every solution converges to $x^\eta$, since $\dot V=0$ only at
rest points. Continuity, positive definiteness, and nonincrease of $V$
give Lyapunov stability, completing the proof.
\end{proof}

\subsection{Universal local and global FW-stability}\label{subsec:ulgfw}
For a cost function $f\in\D(\Rp)$, write
\[
\mathrm{Aff}_+(f):=\{\alpha f+\beta:\alpha>0,\ \beta\in\R\}
\]
for its positive affine hull.  A set $\C\subseteq\D(\Rp)$ is called
\emph{PAT-closed} if $\mathrm{Aff}_+(f)\subseteq\C$ for every $f\in\C$.
Closure of $\C$ under positive affine transformations is needed
to show necessity below because the
rest-point equations depend on the magnitude of the costs, whereas the curvature
condition defining $Q_n(\Rp)$ depends only on $f'$ and $f''$.

\begin{definition}[Universal local and global FW-stability]
\label{def:q-fw}
Let $n\ge2$ and let $\mathcal C\subseteq\mathcal D(\mathbb R_+)$ be
PAT-closed.
We say that $\mathcal C$ is \emph{universally locally FW-stable} if, for
every $\eta>0$ and every $n$-player atomic splittable congestion game with
resource costs in $\mathcal C$, every regularized equilibrium of the
Euclidean-regularized FW dynamics is locally Lyapunov stable.
We say that $\mathcal C$ is \emph{universally globally FW-stable} if, for
every $\eta>0$ and every $n$-player atomic splittable congestion game with
resource costs in $\mathcal C$, the regularized equilibrium set is globally
asymptotically stable.
\end{definition}
\subsection{Local Stability Behavior}\label{subsec:local}
We first differentiate the response near an interior regularized equilibrium.
\begin{lemma}\label{lem:derivative-beta}
Let \(x^*\in\relint(K)\) be a rest point (regularized equilibrium) of \eqref{eq:fw-reg}.  Then
\(\Phi^\eta\) is continuously differentiable in a neighbourhood of \(x^*\) and
\[
 D\Phi^\eta(x^*)|_T
 =
 -\frac1\eta DF(x^*)|_T,
\]
where $T=T_K$ and $DF(x^*)|_T=P_TDF(x^*)P_T$ is the tangent-space
compression introduced in Section~\ref{sec:jacobian}.  Consequently,
\begin{equation}\label{eq:linearization-reg-fw}
 DG(x^*)|_T
 =
 -I-\frac1\eta DF(x^*)|_T .
\end{equation}
\end{lemma}

\begin{proof}
\label{sec:apx-lemmabeta}
Since $x^*\in\relint K$ and $\beta^\eta(F(x^*))=x^*$, Lipschitz
continuity of $\beta^\eta$ ensures that, for $z$ near $F(x^*)$, its
minimizer lies in a relative neighbourhood of $x^*$ contained in $K$.
The minimizer is therefore stationary along $T=\lin(K-K)$ and is the
unique minimizer over $\aff K=x^*+T$. Equivalently, locally
\[
 \beta^\eta(z)=\operatorname{proj}_{x^*+T}(-z/\eta)
 =x^*+P_T(-z/\eta-x^*),\qquad
 D\beta^\eta(z)=-\frac1\eta P_T.
\]
The chain rule for $\Phi^\eta=\beta^\eta\circ F$ gives
$D\Phi^\eta(x^*)h=-P_TDF(x^*)h/\eta$ for $h\in T$.
Taking the tangent-space compression and using $G=\Phi^\eta-\id$
proves both claims.
\end{proof}

\begin{lemma}[Reduction to a strictly complementary face]
\label{lem:fw-face-reduction}
Let $K$ be a compact polytope, let $\mathcal F$ be a face of $K$, and let
$x^*\in\relint(\mathcal F)$ be a rest point of~\eqref{eq:fw-reg}. Write
$T_{\mathcal F}:=\operatorname{lin}(\mathcal F-\mathcal F)$ and assume strict
complementarity in the form
\begin{equation}
\label{eq:fw-face-strict-complementarity}
 -\frac1\eta F(x^*)-x^*
 \in \relint N_K(\mathcal F),
\end{equation}
where $N_K(\mathcal F)$ is the normal cone of $K$ at any point of
$\relint(\mathcal F)$. Then, in a neighbourhood of $F(x^*)$, the minimizer
$\beta^\eta(z)$ lies in $\relint(\mathcal F)$ and is unchanged if $K$ is
replaced by $\mathcal F$ in~\eqref{eq:regularized-best-response}. Consequently,
the restriction of the vector field to $\mathcal F$ is continuously
differentiable near $x^*$ and
\begin{equation}
\label{eq:linearization-reg-fw-face}
 DG(x^*)|_{T_{\mathcal F}}
 =-I-\frac1\eta
 P_{T_{\mathcal F}}DF(x^*)P_{T_{\mathcal F}}.
\end{equation}
If this restricted linearization has an eigenvalue with positive real part,
then $x^*$ is locally unstable for the dynamics on $K$.
\end{lemma}

\begin{proof}
Put $q^*:=-F(x^*)/\eta$. The rest-point condition gives
$x^*=\operatorname{proj}_K(q^*)$, while
\eqref{eq:fw-face-strict-complementarity} places $q^*$ in the relative
interior of the polyhedral projection cell
$\mathcal F+N_K(\mathcal F)$.
Hence the projection of every point sufficiently close to $q^*$ belongs to
$\relint(\mathcal F)$. On this neighbourhood, projection onto $K$ therefore
coincides with projection onto $\mathcal F$, and locally with projection onto
$\operatorname{aff}(\mathcal F)$. Its derivative is
$P_{T_{\mathcal F}}$. Applying the chain rule to
$\Phi^\eta=\beta^\eta\circ F$ yields
\eqref{eq:linearization-reg-fw-face}, exactly as in
Lemma~\ref{lem:derivative-beta}.

Moreover, as long as a trajectory starting in $\mathcal F$ remains in this
neighbourhood, both $x$ and $\Phi^\eta(x)$ lie in $\mathcal F$. Thus the face
is locally invariant. A positive-real-part eigenvalue makes $x^*$ unstable
for the smooth dynamics restricted to $\mathcal F$, and instability along
initial conditions in $\mathcal F\subseteq K$ implies instability in $K$.
\end{proof}

\begin{corollary}[Spectral criterion for regularized FW]
\label{cor:fw-spectral-criterion}
\emph{Relative-interior case.}
Let $K$ be a nonempty compact convex strategy set, put $T=T_K$, and let
\(x^*\in\relint(K)\) be a rest point of \eqref{eq:fw-reg}. Let
\(\mu\) be an eigenvalue of \(DF(x^*)|_T\).  Then
$ -1-\frac{\mu}{\eta}$
is an eigenvalue of \(DG(x^*)|_T\).  In particular, if
\(DF(x^*)|_T\) has an eigenvalue \(\mu\) with \(\Repart(\mu)<-\eta\), then
\(x^*\) is linearly unstable, and hence unstable in the sense of Lyapunov
for the Euclidean-regularized Frank--Wolfe dynamics. If all eigenvalues of
\(DF(x^*)|_T\) have nonnegative real part, then \(x^*\) is locally
exponentially stable.

\emph{Face case.}
Let $K$ be a compact polytope, let $\mathcal F$ be a proper face of $K$,
and let $x^*\in\relint(\mathcal F)$ be a rest point of~\eqref{eq:fw-reg}
satisfying the strict-complementarity
condition~\eqref{eq:fw-face-strict-complementarity}. Then the same spectral
and stability conclusions hold with $T$ replaced by $T_{\mathcal F}$ and
$DF(x^*)|_{T_{\mathcal F}}:=
P_{T_{\mathcal F}}DF(x^*)P_{T_{\mathcal F}}$, where stability and instability
are relative to the full strategy set $K$.
\end{corollary}

\begin{proof}
The relative-interior case follows from \eqref{eq:linearization-reg-fw}.
For the face case, use Lemma~\ref{lem:fw-face-reduction}.
For stability in the full polytope, decompose $T_K$ into
$T_{\mathcal F}$ and its orthogonal complement in $T_K$. Since the local
response lies in $\mathcal F$, its derivative has range in
$T_{\mathcal F}$. The full FW linearization is therefore block triangular,
with the stated face linearization and $-I$ on the complementary diagonal
block.
If \(\Repart(\mu)<-\eta\), then
$\Repart(-1-\mu/\eta)>0$,
so the linearization has an eigenvalue with positive real part.  Conversely, if
\(\Repart(\mu)\ge0\) for every eigenvalue \(\mu\) of \(DF(x^*)|_T\), then every
eigenvalue of \(DG(x^*)|_T\) has real part at most \(-1\).  The claims follow
from the standard linearization theorem for smooth finite-dimensional ODEs; see,
for example, Khalil~\cite[Theorem~4.13]{Khalil2002} or
Perko~\cite[Section~2.9]{Perko2001}.
\end{proof}

\subsection{Proof of Theorem~\ref{thm:fw-characterization}}\label{subsec:proof}
\begin{proof}[Proof of Theorem~\ref{thm:fw-characterization}]

We first prove \ref{enum:mon}.$\Rightarrow$\ref{enum:ugs}. Let
\(\mathcal C\subseteq Q_n(\mathbb R_+)\). By Theorem~\ref{thm:classification}, every
\(n\)-player game with costs in \(\mathcal C\) has a monotone VI operator
\(F\). Proposition~\ref{prop:RFW} gives global asymptotic stability.
In particular, it follows that the unique regularized
equilibrium is locally stable proving also \ref{enum:mon}.$\Rightarrow$\ref{enum:uls}.

We now prove the necessity statement \ref{enum:uls}.$\Rightarrow$\ref{enum:mon}.
Suppose that \(\mathcal C\) is PAT-closed and universally locally FW-stable.
Assume by contradiction that
\(\mathcal C\not\subseteq Q_n(\mathbb R_+)\) and  choose
\(f\in\mathcal C\setminus Q_n(\mathbb R_+)\). 
For this $f$, we will construct a game with interval-based strategy spaces and costs from the set
$\mathrm{Aff}_+(f)$ having an interior regularized equilibrium (hence a rest point of the dynamic)
for which the dynamic is linearly unstable
implying local instability leading to the contradiction $f\notin \C$.  

As \(f\in\D(\Rp)\setminus Q_n(\Rp)\), there exists \(s>0\) such that
\[
 s f''(s)>\gamma_n f'(s),
 \qquad
 \gamma_n=\frac{2(n+1)}{n-1}.
\]
Choose \(\delta>0\) sufficiently small such that \(\delta<1/(n+1)\) and
\begin{equation}\label{eq:general-negative-mu-condition}
 M_\delta
 :=
 \frac{2(n+1)f'(s)-s(n-1)f''(s)+s(n^2-1)\delta f''(s)}{n+3}
 <0.
\end{equation}
Let \(\alpha>0\), choose \(\eta>0\) with
$0<\eta<-\alpha M_\delta$,
and define the positive affine transform
\[
 c(t):=\alpha f(t)+\beta,
 \qquad
 \beta:=s\bigl(\alpha f'(s)+\eta\bigr)(1-(n+1)\delta)-\alpha f(s).
\]
Then
\begin{equation}\label{eq:rest-shift-condition-general}
 c(s)=s\bigl(c'(s)+\eta\bigr)(1-(n+1)\delta),
 \qquad
 c'(s)=\alpha f'(s),
 \qquad
 c''(s)=\alpha f''(s).
\end{equation}

We now construct an \(n\)-player game with \(n\) resources.  Let
\(N=E=\{1,\dots,n\}\), and define
\begin{equation}\label{eq:fw-counterexample-profile-directions}
 x^*_{i,e}
 :=
 \begin{cases}
 s(1-(n-1)\delta), & e=i,\\
 s\delta, & e\ne i,
 \end{cases}
 \qquad
 q_{i,e}
 :=
 \begin{cases}
 -\frac{n-1}{2}, & e=i,\\
 1, & e\ne i.
 \end{cases}
\end{equation}
For every resource \(e\), the load at \(x^*\) is equal to \(s\).  For
\(\varepsilon>0\) small enough, the line segments
\begin{equation}\label{eq:fw-counterexample-strategy-segments}
 K_i:=\{x_i^*+\theta_i q_i:\ -\varepsilon\le\theta_i\le\varepsilon\}
\end{equation}
are contained in \(\mathbb R_{++}^n\), and \(x_i^*\in\relint(K_i)\).  Put cost \(c\) on
every resource.

We first verify that \(x^*\) is a rest point (i.e., a regularized equilibrium).  Since
\[
 \sum_{e=1}^n q_{i,e}=\frac{n-1}{2}
 \qquad\text{and}\qquad
 \sum_{e=1}^n q_{i,e}x^*_{i,e}
 =
 -\frac{s(n-1)}{2}(1-(n+1)\delta),
\]
we have, by \eqref{eq:rest-shift-condition-general},
\[
 \left\langle F_i(x^*)+\eta x_i^*,q_i\right\rangle
 =
 \frac{n-1}{2}c(s)
 -\frac{s(n-1)}{2}\bigl(c'(s)+\eta\bigr)(1-(n+1)\delta)
 =0.
\]
Thus \(x_i^*\) satisfies the first-order condition for minimizing
\(\langle F_i(x^*),y_i\rangle+\eta\norm{y_i}^2/2\) over the line segment
\(K_i\).  Since the objective is strictly convex, \(x_i^*=\beta_i^\eta(F_i(x^*))\)
for every player \(i\), and therefore \(x^*\) is a regularized equilibrium of
\eqref{eq:fw-reg}.

It remains to prove instability.  The tangent space is
\(T_i=\spanop\{q_i\}\) for each player.  

Let \(h=(h_1,\ldots,h_n)\in T\). Since each tangent space \(T_i\) is one-dimensional, we may write
$h_i=\theta_i q_i$, $i=1,\ldots,n$.
In these coordinates, the restriction \(D F(x^*)|_T\) is represented by the matrix
\(B=(B_{ij})_{i,j=1}^n\) defined by
\[
\bigl(P_TDF(x^*)h\bigr)_i
=
\left(\sum_{j=1}^n B_{ij}\theta_j\right)q_i .
\]
Equivalently, \(B_{ij}\) is the scalar coefficient with which a perturbation of
player \(j\) in direction \(q_j\) contributes to the \(i\)-th tangent direction
\(q_i\).
We now verify that \(\mathbf 1=(1,\ldots,1)^\top\) is an eigenvector of \(B\).
Indeed, by substituting \eqref{eq:DF-action-prelim}, the explicit
equilibrium \(x^*\), and the tangent directions \(q_i\), we obtain for every
\(i=1,\ldots,n\)
\[
        \sum_{j=1}^n B_{ij}
        =
        \alpha\,
        \frac{
        2(n+1) f'(s)
        -
        s(n-1) f''(s)
        +
        s(n^2-1)\delta f''(s)
        }{n+3}.
\]
The right-hand side is independent of \(i\). Hence all row sums of \(B\) are
equal, and therefore $B\mathbf 1=\mu\mathbf 1$,
where
\[
        \mu
        =
        \alpha M_\delta
        =
        \alpha\,
        \frac{
        2(n+1) f'(s)
        -
        s(n-1) f''(s)
        +
        s(n^2-1)\delta f''(s)
        }{n+3}.
\]
Thus \(\mathbf 1\) is an eigenvector of the restricted Jacobian
\(D F(x^*)|_T\) with eigenvalue \(\mu=\alpha M_\delta\).

By construction,
\(\mu< -\eta\).  Corollary~\ref{cor:fw-spectral-criterion} therefore implies
that \(x^*\) is linearly unstable.  Hence the constructed game has a regularized equilibrium that is not
Lyapunov stable. This contradicts the assumption that \(\mathcal C\) is
universally locally FW-stable.

It remains to prove \ref{enum:ugs}.$\Rightarrow$\ref{enum:mon}., that is, necessity for universal global FW-stability. Suppose
instead that \(\mathcal C\) is universally globally FW-stable and
\(\mathcal C\not\subseteq Q_n(\mathbb R_+)\). We use the same construction
as above, choosing \(\eta\) outside a finite exceptional set so that the
unstable regularized equilibrium \(x^*\) is isolated.

To this end, let $A:=DF(x^*)|_T$.
The matrix $A$ is finite-dimensional and independent of $\eta$: only the
additive cost shifts in the construction depend on $\eta$, and these do not
affect $DF(x^*)$. We may therefore choose the regularization
parameter $\eta$ above outside the finite set
\[
\mathcal E(A)
:=
\{-\lambda:\lambda\in\operatorname{spec}(A)\cap(-\infty,0)\}.
\]
Indeed, this is possible because $(0,-\mu)$ is a nonempty interval and
$\mathcal E(A)$ contains only finitely many points. With this choice,
\[
DG(x^*)|_T=-I-\frac1\eta A
\]
has no zero eigenvalue and still has an eigenvalue with positive real
part.
Consequently, $DG(x^*)|_T$ is nonsingular. Since $x^*\in\operatorname{relint}(K)$,
the inverse function theorem applied on the tangent space $T$ implies
that $x^*$ is an \emph{isolated rest point}: there exists a relative
neighbourhood $U$ of $x^*$ such that $x^*$ is the only rest point in $U$.
Because $x^*$ is both isolated and Lyapunov unstable, the rest-point set
cannot be Lyapunov stable. Indeed, choose $r>0$ such that
$\mathcal R_\eta\cap B_r(x^*)=\{x^*\},$
where $\mathcal R_\eta$ denotes the rest-point set of the regularized FW
dynamics. Since $x^*$ is Lyapunov unstable, there are initial conditions
arbitrarily close to $x^*$ whose trajectories leave $B_{r/2}(x^*)$.
At the first exit time, the trajectory has distance at least $r/2$ from
the whole rest-point set $\mathcal R_\eta$. Hence $\mathcal R_\eta$ is
not Lyapunov stable and therefore cannot be globally asymptotically stable.
Thus, if a PAT-closed class $\mathcal C$ has either the universal local
or the universal global FW-stability property, then necessarily
$\mathcal C\subseteq Q_n(\mathbb R_+)$. This proves the characterizations of the theorem.

Finally, $Q_n(\Rp)$ is PAT-closed: for $c=\alpha f+\beta$ with
$\alpha>0$, we have $\Delta_c=\alpha\Delta_f$ and
$\Theta_c=\alpha\Theta_f$. It is therefore the unique largest class
with either universal FW-stability property.
\end{proof}
We now derive an explicit instability result for
games with $3$ players and polynomial costs of maximum degree $6$.
 For this case, the maximum degree cannot be further
reduced within the class of polynomials with nonnegative
coefficients: when \(n=3\), the curvature bound is
\(x f''(x)\le 4 f'(x)\), and every such polynomial of degree at most five
satisfies this bound term by term, see Remark~\ref{ref:rem-costs}.

\begin{proof}[Proof of Corollary~\ref{cor:degree-six-counterexample}]
Let \(N=E=\{1,2,3\}\), set
\[
 \delta:=\frac1{100},
 \qquad
 \eta:=\frac3{100},
 \qquad
 \kappa:=\frac{3083}{3750},
\]
and define the candidate regularized equilibrium \(x^*\) and the line directions \(q_i\) by
\[
 x^*=
 \begin{pmatrix}
  49/50 & 1/100 & 1/100\\
  1/100 & 49/50 & 1/100\\
  1/100 & 1/100 & 49/50
 \end{pmatrix},
 \qquad
 Q=
 \begin{pmatrix}
  -1 & 1 & 1\\
   1 &-1 & 1\\
   1 & 1 &-1
 \end{pmatrix},
\]
where \(q_i\) is the \(i\)-th row of \(Q\).  The resource loads at \(x^*\) are
all equal to one.  Let \(\varepsilon=1/200\) and define
\[
 K_i:=\{x_i^*+\theta_i q_i:\ -\varepsilon\le\theta_i\le\varepsilon\}.
\]
Since every entry of \(x^*\) is at least \(1/100\), these line segments are
contained in \(\R_+^3\), and \(x_i^*\in\relint(K_i)\).

For every resource let $c_e(t)=\kappa+\frac16 t^6.$
Then \(c_e(1)=\kappa+1/6\), \(c'_e(1)=1\), and \(c''_e(1)=5\).  Since row
\(q_i\) has entry \(-1\) in coordinate \(i\) and entry \(+1\) in the two other
coordinates, we have, for every player \(i\),
\begin{align*}
 \left\langle F_i(x^*)+\eta x_i^*,q_i\right\rangle
 &=
 \sum_{e=1}^3 q_{i,e}
 \left(\kappa+\frac16+(1+\eta)x_{i,e}^*\right)\\
 &=
 \kappa+\frac16-(1+\eta)(1-4\delta)
 =0,
\end{align*}
because
\[
 \kappa=(1+\eta)(1-4\delta)-\frac16
 =\frac{3083}{3750}.
\]
Thus \(x_i^*\) satisfies the first-order condition for minimizing
\(\langle F_i(x^*),y_i\rangle+\eta\norm{y_i}^2/2\) over \(K_i\).  The objective
is strictly convex on \(K_i\), so \(x_i^*=\beta_i^\eta(F_i(x^*))\) for all
\(i\).  Hence \(x^*\) is a regularized equilibrium of \eqref{eq:fw-reg}.

It remains to check the linearization.  Since \(T_i=\spanop\{q_i\}\), write a
tangent vector as \(h_i=\theta_iq_i\).  In the coordinates
\(\theta=(\theta_1,\theta_2,\theta_3)\), the operator \(DF(x^*)|_T\) is
represented by the matrix \(B\) with entries
\begin{equation}\label{eq:coordinate-jacobian-degree-six-certificate}
 B_{ij}
 =
 \frac1{\norm{q_i}^2}
 \sum_{e=1}^3
 q_{i,e}
 \left(
 \delta_{ij}q_{i,e}+
 \bigl(1+5x^*_{i,e}\bigr)q_{j,e}
 \right).
\end{equation}
Substitution gives
\[
B=
\begin{pmatrix}
 11/3 & -59/30 & -59/30\\
 -59/30 & 11/3 & -59/30\\
 -59/30 & -59/30 & 11/3
\end{pmatrix}.
\]
Therefore the spectrum of $B$ is given by
\[
 \spec(B)=\left\{-\frac4{15},\frac{169}{30},\frac{169}{30}\right\}.
\]
In particular, \(B\) has the eigenvalue
\(\mu=-4/15<-3/100=-\eta\).  By
Corollary~\ref{cor:fw-spectral-criterion}, the linearized
Euclidean-regularized FW dynamics has the corresponding eigenvalue
\[
 -1-\frac{\mu}{\eta}
 =
 -1+\frac{4/15}{3/100}
 =\frac{71}{9}>0.
\]
Thus \(x^*\) is linearly, and hence locally, unstable.
\end{proof}

\section{Characterizing Uniqueness}
\label{sec:uniqueness-characterization}

We now prove the uniqueness characterization stated in
Theorem~\ref{thm:uniqueness-characterization}. The point is that the construction used in the proof of
Theorem~\ref{thm:fw-characterization} to obtain a negative eigenvalue also
gives non-uniqueness of Nash equilibria. The additional ingredient is the
degree of the natural map of a variational inequality.

\subsection{Degree and local index}
\label{sec:brouwer-degree}
\label{def:brouwer-degree}
For a bounded open $\Omega\subset\mathbb R^d$, a continuous map
$H:\overline\Omega\to\mathbb R^d$, and
$y\notin H(\partial\Omega)$, the Brouwer degree
$\deg(H,\Omega,y)\in\mathbb Z$ is characterized by normalization,
additivity (hence excision), and invariance under homotopies that avoid
$y$ on the boundary. If $H$ is $C^1$ and $y$ is a regular value with
finitely many preimages in $\Omega$, then
\begin{equation}\label{eq:brouwer-degree-regular-value}
 \deg(H,\Omega,y)
 =\sum_{x\in H^{-1}(y)\cap\Omega}\operatorname{sign}\det DH(x).
\end{equation}
The local index of an isolated zero $x^*$ is
$\operatorname{ind}(H,x^*)=\deg(H,U,0)$ for a sufficiently small
isolating neighbourhood $U$; for a nondegenerate zero it equals
$\operatorname{sign}\det DH(x^*)$.
Thus a negative-index zero cannot be the only zero of a map of total
degree one, by excision. We use these standard facts for the natural
map of a variational inequality; see
Facchinei and Pang~\cite[Section~2.1 and Proposition~4.2.7]{facchinei2003}.

\subsection{Proof of the uniqueness characterization}
\begin{proof}[Proof of Theorem~\ref{thm:uniqueness-characterization}]
The sufficiency direction follows from the strict-monotonicity
characterization in Theorem~\ref{thm:classification}. Indeed, if all costs
belong to \(Q_n(\Rp)\cap\D_{\rm mc-si}(\Rp)\), then the marginal cost operator
is strictly monotone on every nontrivial feasible direction. Hence the
associated variational inequality has at most one solution. Existence follows
from continuity and compactness of the strategy space. Thus the Nash
equilibrium is unique.

We prove the two necessity directions separately.

\medskip
\noindent
\emph{Step 1: necessity of \(Q_n(\Rp)\).}
Suppose that \(f\in\C\setminus Q_n(\Rp)\). We reuse the line-segment
construction from the necessity proof of
Theorem~\ref{thm:fw-characterization}, now with regularization parameter
\(\eta=0\). Thus, choose \(s>0\) and \(0<\delta<1/(n+1)\) so that
\(M_\delta<0\) in~\eqref{eq:general-negative-mu-condition}, let
\(N=E=\{1,\dots,n\}\), and take \(x^*\), \(q_i\), and \(K_i\) from
\eqref{eq:fw-counterexample-profile-directions}--\eqref{eq:fw-counterexample-strategy-segments}.
In particular, every resource has load \(s\) at \(x^*\),
\(x_i^*\in\relint(K_i)\), and \(T_i=\spanop\{q_i\}\).

Choose \(\alpha>0\), and define
\[
 \beta
 =
 \alpha\bigl(sf'(s)(1-(n+1)\delta)-f(s)\bigr),
 \qquad
 c=\alpha f+\beta.
\]
PAT-closedness and the defining derivative conditions for \(\D(\Rp)\) give
\(c\in\C\cap\D(\Rp)\); moreover, \(M_\delta<0\) implies \(c''(s)>0\).
Put cost \(c\) on every resource. The first-order
calculation in the earlier construction, with \(\eta=0\), gives
\(\langle F_i(x^*),q_i\rangle=0\) for every \(i\). Convexity of each player's
objective therefore shows that \(x^*\) is a Nash equilibrium.

We next compute the local index of this equilibrium. Let
\(T=\bigtimes_{i\in N}T_i\). In the coordinates \(h_i=\theta_iq_i\),
permutation symmetry gives a matrix with a common diagonal entry and a common
off-diagonal entry. The row-sum calculation from the earlier construction
gives the eigenvalue
\begin{equation}\label{eq:uniqueness-mu-one}
 \mu_1=\alpha M_\delta<0
\end{equation}
on \(\ones\), while the diagonal-minus-off-diagonal eigenvalue is
\begin{equation}\label{eq:uniqueness-mu-two}
 \mu_2
 =
 c'(s)
 +\frac{(n+1)A}{n+3}
 +\frac{2(n+1)C}{(n-1)(n+3)},
\end{equation}
with multiplicity \(n-1\), where
\[
 A=c'(s)+c''(s)s(1-(n-1)\delta),
 \qquad
 C=c'(s)+c''(s)s\delta.
\]
Since \(c'(s)\ge0\), \(c''(s)>0\),
\(A>0\), and \(C>0\), we have \(\mu_2>0\). Thus
\begin{equation}\label{eq:uniqueness-negative-determinant}
 \det\bigl(DF(x^*)|_T\bigr)
 =\mu_1\mu_2^{n-1}<0.
\end{equation}

It remains to explain why a Nash equilibrium with negative restricted
determinant cannot be unique. We use the Brouwer degree and local index recalled
in Section~\ref{sec:brouwer-degree}. Let \(K=\bigtimes_{i\in N}K_i\), and consider the
natural map of the associated variational inequality,
\[
 F^{\mathrm{nat}}(x)
 =x-\Pi_K(x-F(x)).
\]
Here \(\Pi_K\) denotes the Euclidean projection onto the nonempty closed
convex set \(K\).
Its zeros are exactly the Nash equilibria. Because \(K\) is a compact subset
of the strictly positive orthant, we may choose a bounded open set
\(\Omega\subset\mathbb R_{++}^{N\times E}\) with \(K\subset\Omega\). Then
\begin{equation}\label{eq:natural-map-total-degree}
 \deg(F^{\mathrm{nat}},\Omega,0)=1.
\end{equation}
Indeed, for \(x_0\in K\), the homotopy
\[
 H(x,t)
 =x-\Pi_K\bigl(t(x-F(x))+(1-t)x_0\bigr)
\]
connects \(F^{\mathrm{nat}}\) to \(x\mapsto x-x_0\), and every zero of
\(H(\cdot,t)\) lies in \(K\). Hence no zero crosses \(\partial\Omega\), and
the degree is one.

Since \(x^*\in\relint(K)\), the equilibrium condition gives
\(P_TF(x^*)=0\). Hence the projection of \(x^*-F(x^*)\) onto
\(\aff(K)\) is \(x^*\), and the projection onto \(K\) agrees locally at this
argument with the projection onto \(\aff(K)\). Writing
$T=\lin(K-K)$,
we therefore have, locally around \(x^*\),
\[
 F^{\mathrm{nat}}(x)
 =P_{T^\perp}(x-x^*)+P_TF(x).
\]
Consequently,
\[
 \det DF^{\mathrm{nat}}(x^*)
 =\det\bigl(DF(x^*)|_T\bigr)<0.
\]
Thus, by~\eqref{eq:brouwer-degree-regular-value}, \(x^*\) is a nondegenerate
zero of \(F^{\mathrm{nat}}\) with local index \(-1\). If \(x^*\) were its
only zero, excision would imply
$\deg(F^{\mathrm{nat}},\Omega,0)=-1$,
contradicting~\eqref{eq:natural-map-total-degree}. Hence the constructed game
has at least one further Nash equilibrium. Therefore any PAT-closed class
guaranteeing uniqueness must satisfy \(\C\subseteq Q_n(\Rp)\).

\medskip
\noindent
\emph{Step 2: necessity of \(\D_{\rm mc-si}(\Rp)\).}
Suppose that $f\in\C\cap Q_n(\Rp)$ and
$f\notin\D_{\rm mc-si}(\Rp)$.
Define \(\varphi(x)=f(x)+xf'(x)\). Since
$\varphi'(x)=2f'(x)+xf''(x)=\Theta_f(x)\ge0$,
the function \(\varphi\) is nondecreasing. Because
\(f\notin\D_{\rm mc-si}(\Rp)\), there exists a nondegenerate interval
\([a,b]\subseteq\Rp\) on which \(\varphi\) is constant, say
$\varphi(x)=\gamma$ for all $x\in[a,b]$.
Set \(c=f-\gamma\). By PAT-closedness, \(c\in\C\), and
$c(x)+xc'(x)=0$ for all $x\in[a,b]$.

Now consider a one-resource game. Let player \(1\) have strategy set
\(K_1=[a,b]\), and let all other players have singleton strategy sets
\(K_j=\{0\}\), \(j=2,\dots,n\). Player \(1\)'s cost is
\(\pi_1(x)=xc(x)\), and hence
$\pi_1'(x)=c(x)+xc'(x)=0$ for all $x\in[a,b]$.
Thus \(\pi_1\) is constant on \([a,b]\), and every point of \([a,b]\) is a
Nash equilibrium. Therefore any PAT-closed class guaranteeing uniqueness must
also satisfy \(\C\subseteq\D_{\rm mc-si}(\Rp)\).

Combining the two necessity directions with the sufficiency argument proves
the theorem.
\end{proof}

\begin{remark}
The proof gives two interpretations of the same curvature threshold. On the
dynamical side, violation of
$xf''(x)\le\frac{2(n+1)}{n-1}f'(x)$
creates a negative eigenvalue of the restricted marginal-cost Jacobian,
which yields a positive eigenvalue in the Frank--Wolfe linearization for
the chosen regularization. On the
equilibrium side, the same construction gives an interior equilibrium whose
restricted Jacobian has negative determinant. Since the natural map of a
compact variational inequality has total degree one, this equilibrium has the
wrong index to be unique.
\end{remark}

\section{Network Congestion Games}
\label{sec:network-games}

We prove Theorem~\ref{thm:directed-network-characterizations} using acyclic
directed networks, then establish the stronger parallel-link stability result
in Theorem~\ref{thm:parallel}.

A directed atomic splittable network congestion game consists of a directed
graph $G=(V,E)$, terminals $(s_i,t_i)$, and a demand $d_i>0$ for each player.
We formulate the strategy sets directly in arc-flow variables. Writing $B$ for
the node--arc incidence matrix with the net-outflow convention, player $i$
chooses from the compact edge-flow
polytope
\[
 K_i=\left\{x_i\in[0,d_i]^E:
 Bx_i=d_i(\mathbf e_{s_i}-\mathbf e_{t_i})\right\}.
\]
Thus the graph may contain directed cycles, and feasible flows may contain
circulations. The upper bounds only ensure compactness.

With strictly positive costs, canceling a circulation strictly decreases both
the player's objective and the regularized oracle objective. Thus Nash and
regularized equilibria are cycle-free.

\subsection{FW-Stability on General Networks}
\label{subsec:general-networks}

\begin{proof}[Proof of the FW-stability statement in
Theorem~\ref{thm:directed-network-characterizations}]
The implications from item~\ref{enum:network-fw-class} follow from
Theorem~\ref{thm:fw-characterization}, because network flow polytopes are
compact convex sets.

We prove the converse implications by establishing necessity of $Q_n(\Rp)$.
Let $n\ge2$, let $N=\{1,\ldots,n\}$, and choose
$f\in\C\setminus Q_n(\Rp)$ and $s>0$ such that
\begin{equation}\label{eq:network-curvature-violation}
 sf''(s)>\gamma_n f'(s).
\end{equation}
This implies $f'(s)>0$ and $f''(s)>0$. Introduce two layers of shared
congestible arcs
\[
 r_e^L=(a_e^L,b_e^L),\qquad r_e^R=(a_e^R,b_e^R).
\]
Add sources $s_i$, sinks $t_i$, and the connector arcs
\begin{align*}
 &(s_i,a_e^L) &&(i,e\in N),
 &\qquad &(b_i^L,t_i) &&(i\in N),\\
 &(s_i,a_i^R) &&(i\in N),
 &&(b_e^L,a_k^R) &&(e,k\in N),\\
 &(b_k^R,t_i) &&(i,k\in N).
\end{align*}
All arcs belong to the same graph and are available to every player. The
graph is acyclic, and the $s_i$--$t_i$ paths are precisely
\[
 D_i^L=\{r_i^L\},\qquad D_i^R=\{r_i^R\},\qquad
 P_i^{ek}=\{r_e^L,r_k^R\}\quad(e,k\in N),
\]
where only the congestible arcs are displayed.

\begin{figure}[!ht]
\centering
\begin{tikzpicture}[
  >=stealth,
  vertex/.style={circle,draw,fill=white,minimum size=6.5mm,inner sep=1pt},
  connector/.style={->,draw=black,dashed,line width=0.6pt},
  resource/.style={->,draw=black,line width=1.8pt},
  layerlabel/.style={font=\small\bfseries}
]
  \node[layerlabel] at (2.8,2.65) {left resource layer};
  \node[layerlabel] at (7.0,2.65) {right resource layer};
  \node[vertex] (s1) at (0,1.35) {$s_1$};
  \node[vertex] (s2) at (0,-1.35) {$s_2$};
  \node[vertex] (a1L) at (2.1,1.35) {$a_1^L$};
  \node[vertex] (b1L) at (3.5,1.35) {$b_1^L$};
  \node[vertex] (a2L) at (2.1,-1.35) {$a_2^L$};
  \node[vertex] (b2L) at (3.5,-1.35) {$b_2^L$};
  \node[vertex] (a1R) at (6.3,1.35) {$a_1^R$};
  \node[vertex] (b1R) at (7.7,1.35) {$b_1^R$};
  \node[vertex] (a2R) at (6.3,-1.35) {$a_2^R$};
  \node[vertex] (b2R) at (7.7,-1.35) {$b_2^R$};
  \node[vertex] (t1) at (10,1.35) {$t_1$};
  \node[vertex] (t2) at (10,-1.35) {$t_2$};
%
  % Sources to the left layer.
  \draw[connector] (s1) -- (a1L);
  \draw[connector] (s1) -- (a2L);
  \draw[connector] (s2) -- (a1L);
  \draw[connector] (s2) -- (a2L);
  % Direct access to the right layer.
  \draw[connector] (s1) to[bend left=18] (a1R);
  \draw[connector] (s2) to[bend right=18] (a2R);
  % Complete connector family between the two resource layers.
  \draw[connector] (b1L) -- (a1R);
  \draw[connector] (b1L) -- (a2R);
  \draw[connector] (b2L) -- (a1R);
  \draw[connector] (b2L) -- (a2R);
  % Right layer to both sinks.
  \draw[connector] (b1R) -- (t1);
  \draw[connector] (b1R) -- (t2);
  \draw[connector] (b2R) -- (t1);
  \draw[connector] (b2R) -- (t2);
  % Direct exits from the left layer.
  \draw[connector] (b1L) to[bend left=18] (t1);
  \draw[connector] (b2L) to[bend right=18] (t2);
  \draw[resource] (a1L) -- node[above,font=\small] {$r_1^L$} (b1L);
  \draw[resource] (a2L) -- node[below,font=\small] {$r_2^L$} (b2L);
  \draw[resource] (a1R) -- node[above,font=\small] {$r_1^R$} (b1R);
  \draw[resource] (a2R) -- node[below,font=\small] {$r_2^R$} (b2R);
  \draw[resource] (2.35,-2.55) -- (3.15,-2.55);
  \node[anchor=west,font=\small] at (3.25,-2.55) {congestible arc};
  \draw[connector] (6.55,-2.55) -- (7.35,-2.55);
  \node[anchor=west,font=\small] at (7.45,-2.55) {connector arc};
\end{tikzpicture}
\caption{Product two-layer network for $n=2$. Thin dashed arcs are connectors;
thick solid arcs are congestible resources. For general $n$, each index has
one source, one sink, and one resource in each layer.
\label{fig:network-two-layer-construction}}

\end{figure}

Choose $0<\delta<1/(n+1)$ and put
\[
 A=s(1-(n-1)\delta),\qquad B=s\delta,\qquad
 w=\frac{B}{n-1},\qquad d_i=s+A.
\]
At the candidate profile, player $i$ sends flow $A$ on each direct path,
flow $w$ on each path $P_i^{jk}$ with $j,k\ne i$, and zero flow on all
remaining paths. Each shared resource arc then has total load $s$. The
positive-flow paths define a face of the path-flow polytope.
The path-to-edge map is injective: each two-layer path flow is recovered
from its connector $(b_e^L,a_k^R)$, and the two direct path flows are
recovered from $(b_i^L,t_i)$ and $(s_i,a_i^R)$. Thus this map is an affine
isomorphism onto the edge-flow polytope. Write $\mathcal F$ for the
corresponding face in edge-flow coordinates; the candidate lies
in $\relint(\mathcal F)$.

Put $m=(n-1)/2$, and let $y_i^*$ denote player $i$'s candidate path-flow
vector. Define a feasible perturbation direction $p_i$ in path-flow
coordinates by
\[
 p_{i,D_i^L}=p_{i,D_i^R}=-m,
 \qquad
 p_{i,P_i^{jk}}=\frac1{n-1}\quad(j,k\ne i).
\]
All other entries are zero. For $\theta>0$, the perturbed vector
$y_i^*+\theta p_i$ decreases the flow on each direct path by $m\theta$
and increases the flow on each of the $(n-1)^2$ used two-layer paths by
$\theta/(n-1)$. The total change is zero, since
$-2m+(n-1)^2/(n-1)=0$, so player $i$'s demand is preserved. For
sufficiently small $|\theta|$, the perturbed flows remain in the same face,
because every affected path initially carries positive flow.
On either shared-resource layer, the induced edge-flow change is
$\theta q_i$, where
\[
 q_{i,e}=\begin{cases}-m,&e=i,\\1,&e\ne i.\end{cases}
\]
The corresponding line-segment construction has the critical eigenvalue
\begin{equation}\label{eq:network-critical-eigenvalue}
 \mu_0=\alpha M_\delta(f;s),\qquad
 M_\delta(f;s):=
 \frac{2(n+1)f'(s)-s(n-1)f''(s)
       +s(n^2-1)\delta f''(s)}{n+3}.
\end{equation}
At $\delta=0$,
\[
 M_0(f;s)=\frac{n-1}{n+3}
 \bigl(\gamma_n f'(s)-sf''(s)\bigr)<0,
\]
so choose such a $\delta$. Put the same cost $c=\alpha f+\beta$ on the
shared resource arcs. To identify the desired rest point and its spectrum,
first consider the auxiliary limiting instance in which every connector has
zero cost; at this stage we do not require the connector costs to belong to
$\C$. For a fixed $\eta>0$, choose the vertical shift so that
\begin{equation}\label{eq:network-cost-shift}
 c(s)=(A-2B)c'(s)+\eta(3A-4B-w).
\end{equation}
A direct path then has the same regularized first-order score as every
positive-flow two-layer path. A path with exactly one resource index equal to
$i$ has a strictly larger score, by
\[
 (A-B)c'(s)+\eta\bigl(2(A-B)-w\bigr)>0,
\]
and the path $P_i^{ii}$ has the strictly positive score gap
\[
 2(A-B)c'(s)+\eta\bigl(4(A-B)-w\bigr)>0.
\]
Thus the candidate is a regularized equilibrium, and the strict score gaps
are precisely the strict-complementarity condition for $\mathcal F$. By
Lemma~\ref{lem:fw-face-reduction}, the regularized response agrees locally
with the response obtained by replacing the full strategy polytope by
$\mathcal F$. In particular, an unstable direction tangent to $\mathcal F$
implies instability in the full strategy polytope.

Let $L$ be the block-diagonal path-to-edge map for the used paths, and
let $T_p$ be their path-flow tangent space, with the demand constraint
$\sum_P h_{i,P}=0$ for each player. Write $L_p:=L|_{T_p}$ for the resulting
isomorphism from $T_p$ onto $T_{\mathcal F}$. The regularizer in the network
game is defined on edge flows. With $x=Ly$, its expression in path
coordinates is therefore
\[
 \frac\eta2\|x\|^2=\frac\eta2\|Ly\|^2
 =\frac\eta2 y^\top L^\top Ly,
\]
not, in general, $\frac\eta2\|y\|^2$. In particular, every connector arc
contributes to this squared edge-flow norm, even when its congestion cost
is zero. Define the path-cost Jacobian and the regularization Hessian by
\[
 J=L^\top DF(x^*)L,\qquad
 H=L^\top L=\nabla_y^2\left(\tfrac12\|Ly\|^2\right),
\]
and let $J_p=P_{T_p}J|_{T_p}$ and $H_p=P_{T_p}H|_{T_p}$, where
$P_{T_p}$ denotes orthogonal projection onto $T_p$. Since $L_p$ is
injective, $H_p$ is positive definite.

To relate these matrices to the edge-flow dynamics, set
\[
 \mathcal J:=P_{T_{\mathcal F}}DF(x^*)|_{T_{\mathcal F}}.
\]
This operator applies the marginal-cost Jacobian to a tangent direction
and projects the result back onto $T_{\mathcal F}$. Its representation in
path coordinates is $H_p^{-1}J_p$: indeed,
$J_p=L_p^\top\mathcal J L_p$ and $H_p=L_p^\top L_p$, so
\begin{equation}\label{eq:network-path-coordinate-jacobian}
 L_p^{-1}\mathcal J L_p=H_p^{-1}J_p.
\end{equation}
Thus Lemma~\ref{lem:fw-face-reduction} gives the FW linearization in path
coordinates as
\begin{equation}\label{eq:network-fw-path-metric}
 -I-\frac1\eta H_p^{-1}J_p.
\end{equation}
The factor $H_p^{-1}$ accounts for the change of coordinates.

For completeness, write $a=c'(s)$ and $b=c''(s)$. The direction $p$ above
has aggregate variation $m=(n-1)/2$ on every shared resource arc.
Using~\eqref{eq:DF-action-prelim}, the entries of $Jp$ and $Hp$ are
\[
\begin{array}{c|cc}
 & D_i^L,D_i^R & P_i^{jk}\ (j,k\ne i)\\ \hline
 Jp & mAb & (n+1)a+2mBb\\
 Hp & -3m & 4+1/(n-1).
\end{array}
\]
Projecting off the constant path vector for each player yields
\[
 J_p p=\kappa_n^F\mu_0p,\qquad H_p p=\kappa_n^Hp,
\]
where
\[
 \kappa_n^F=\frac{(n-1)(n+3)}{(n-1)^2+2},\qquad
 \kappa_n^H=\frac{3n^2+2n-3}{(n-1)^2+2}.
\]
Consequently,
\[
 H_p^{-1}J_pp=\frac{\kappa_n^F}{\kappa_n^H}\mu_0p.
\]
By~\eqref{eq:network-path-coordinate-jacobian}, the projected edge-flow
Jacobian $\mathcal J$ has the same eigenvalue, with eigenvector $Lp$.
Hence the FW linearization~\eqref{eq:network-fw-path-metric} has eigenvalue
\[
 -1-\rho_n\frac{\alpha M_\delta(f;s)}{\eta},\qquad
 \rho_n:=\frac{\kappa_n^F}{\kappa_n^H}
 =\frac{(n-1)(n+3)}{3n^2+2n-3}>0.
\]
Choosing $\alpha$ sufficiently large makes this eigenvalue positive.
To ensure nondegeneracy, fix an orthonormal basis of $T_p$. With zero-cost
connectors, write $J_p=\alpha J_{p,0}$; the vertical shift of the shared
costs does not affect this Jacobian. The derivative of the regularized face
stationarity equations is $\alpha J_{p,0}+\eta H_p$. Its determinant is a
polynomial in $\alpha$ and is positive at $\alpha=0$, since $H_p$ is
positive definite. There are therefore only finitely many exceptional
positive scalings. Choose a sufficiently large nonexceptional $\alpha$.
The full FW linearization is then nonsingular: its complementary block
outside the face is $-I$, as in
Corollary~\ref{cor:fw-spectral-criterion}.

Finally, replace every zero connector cost by $\varepsilon f$, which
belongs to $\operatorname{Aff}_+(f)$ for $\varepsilon>0$. On the compact
range of feasible loads these costs converge to zero in $C^2$ as
$\varepsilon\downarrow0$. The implicit function theorem applied to the
face stationarity equations gives a nearby regularized equilibrium. For
small enough $\varepsilon$, its used path flows remain positive, its strict
unused-path gaps persist, and its linearization remains nonsingular with
an eigenvalue of positive real part. Thus it is a strictly complementary,
isolated rest point of the full network game. PAT-closedness ensures that
all its arc costs belong to $\C$.
The rest point is linearly unstable, proving
that item~\ref{enum:network-fw-local} implies
item~\ref{enum:network-fw-class}. Since it is isolated, the rest-point set is
not Lyapunov stable, proving that item~\ref{enum:network-fw-global} implies
item~\ref{enum:network-fw-class}.
\end{proof}

\subsection{Uniqueness on General Networks}
\label{subsec:network-uniqueness}

The product network used for FW instability is not directly suitable for the
degree argument. Indeed, for $n\ge3$ its cross paths admit alternating
transportation directions whose row and column sums vanish. With zero-cost
connectors these directions lie in the kernel of the restricted marginal-cost
Jacobian, so its determinant vanishes. For uniqueness we therefore use a
diagonal two-layer network that removes this kernel.

For the boundary index argument, extend each cost across zero by
\[
 \widetilde c(t)=
 \begin{cases}
 c(t),&t\ge0,\\
 c(0)+c'(0)t+\tfrac12c''(0)t^2,&t<0.
 \end{cases}
\]
This gives a $C^1$ ambient extension of $F$; no convexity or monotonicity
assumptions are needed outside the feasible set.

\begin{lemma}[Index on a strictly complementary face]
\label{lem:face-index}
Let $K$ be a nonempty compact polytope, let $F$ be continuously
differentiable on an open neighbourhood of $K$, and let $\mathcal F$ be a
face of $K$. Suppose $x^*\in\relint(\mathcal F)$ solves
$\operatorname{VI}(K,F)$. Assume strict
complementarity,
\[
 -F(x^*)\in\relint N_K(\mathcal F),
\]
and write $T_{\mathcal F}=\operatorname{lin}(\mathcal F-\mathcal F)$. If
\[
 \det\bigl(P_{T_{\mathcal F}}DF(x^*)P_{T_{\mathcal F}}
       \big|_{T_{\mathcal F}}\bigr)<0,
\]
then $\operatorname{VI}(K,F)$ has at least two solutions.
\end{lemma}

\begin{proof}
The natural map $F^{\rm nat}(x)=x-\Pi_K(x-F(x))$ has the VI solutions as
its zeros and total degree one by the proof of
Theorem~\ref{thm:uniqueness-characterization}. By the projection-cell argument
of Lemma~\ref{lem:fw-face-reduction}, strict complementarity makes projection
onto $K$ agree locally with projection onto $\operatorname{aff}(\mathcal F)$
near $x^*-F(x^*)$. Thus, near $x^*$,
\[
 F^{\rm nat}(x)
 =P_{T_{\mathcal F}^{\perp}}(x-x^*)+P_{T_{\mathcal F}}F(x).
\]
In the decomposition $T_{\mathcal F}\oplus T_{\mathcal F}^{\perp}$, its
derivative is block triangular with the compressed Jacobian and an identity
block on the diagonal. Its negative determinant gives local index $-1$, so
there must be another zero; see Section~\ref{sec:brouwer-degree}.
\end{proof}

For $n\ge2$, take the product network from
Section~\ref{subsec:general-networks}, with $N=\{1,\ldots,n\}$, and delete
every interlayer connector $(b_e^L,a_k^R)$ with $e\ne k$. All resource arcs
and source/sink connectors remain unchanged. The resulting acyclic graph has
exactly the $s_i$--$t_i$ paths
\[
 D_i^L=\{r_i^L\},\qquad D_i^R=\{r_i^R\},\qquad
P_i^e=\{r_e^L,r_e^R\}\quad(e\in N),
\]
where only congestible arcs are displayed.

Figure~\ref{fig:network-diagonal-uniqueness} illustrates this diagonal
construction for two players.
\begin{figure}[!ht]
\centering
\begin{tikzpicture}[
  >=stealth,
  vertex/.style={circle,draw,fill=white,minimum size=6.5mm,inner sep=1pt},
  connector/.style={->,draw=black,dashed,line width=0.6pt},
  resource/.style={->,draw=black,line width=1.8pt},
  layerlabel/.style={font=\small\bfseries}
]
  \node[layerlabel] at (2.8,2.65) {left resource layer};
  \node[layerlabel] at (7.0,2.65) {right resource layer};
  \node[vertex] (s1) at (0,1.35) {$s_1$};
  \node[vertex] (s2) at (0,-1.35) {$s_2$};
  \node[vertex] (a1L) at (2.1,1.35) {$a_1^L$};
  \node[vertex] (b1L) at (3.5,1.35) {$b_1^L$};
  \node[vertex] (a2L) at (2.1,-1.35) {$a_2^L$};
  \node[vertex] (b2L) at (3.5,-1.35) {$b_2^L$};
  \node[vertex] (a1R) at (6.3,1.35) {$a_1^R$};
  \node[vertex] (b1R) at (7.7,1.35) {$b_1^R$};
  \node[vertex] (a2R) at (6.3,-1.35) {$a_2^R$};
  \node[vertex] (b2R) at (7.7,-1.35) {$b_2^R$};
  \node[vertex] (t1) at (10,1.35) {$t_1$};
  \node[vertex] (t2) at (10,-1.35) {$t_2$};
%
  % Sources to the left layer.
  \draw[connector] (s1) -- (a1L);
  \draw[connector] (s1) -- (a2L);
  \draw[connector] (s2) -- (a1L);
  \draw[connector] (s2) -- (a2L);
  % Direct access to the matching right resource.
  \draw[connector] (s1) to[bend left=18] (a1R);
  \draw[connector] (s2) to[bend right=18] (a2R);
  % Only diagonal connectors between the resource layers.
  \draw[connector] (b1L) -- (a1R);
  \draw[connector] (b2L) -- (a2R);
  % Right layer to both sinks.
  \draw[connector] (b1R) -- (t1);
  \draw[connector] (b1R) -- (t2);
  \draw[connector] (b2R) -- (t1);
  \draw[connector] (b2R) -- (t2);
  % Direct exits from the matching left resource.
  \draw[connector] (b1L) to[bend left=18] (t1);
  \draw[connector] (b2L) to[bend right=18] (t2);
  \draw[resource] (a1L) -- node[above,font=\small] {$r_1^L$} (b1L);
  \draw[resource] (a2L) -- node[below,font=\small] {$r_2^L$} (b2L);
  \draw[resource] (a1R) -- node[above,font=\small] {$r_1^R$} (b1R);
  \draw[resource] (a2R) -- node[below,font=\small] {$r_2^R$} (b2R);
  \draw[resource] (2.35,-2.55) -- (3.15,-2.55);
  \node[anchor=west,font=\small] at (3.25,-2.55) {congestible arc};
  \draw[connector] (6.55,-2.55) -- (7.35,-2.55);
  \node[anchor=west,font=\small] at (7.45,-2.55) {connector arc};
\end{tikzpicture}
\caption{Diagonal two-layer network for $n=2$. Thin dashed arcs are connectors;
thick solid arcs are congestible resources. Only equal-index resources are
linked between layers, eliminating the transportation kernel.
\label{fig:network-diagonal-uniqueness}}

\end{figure}

The path decomposition is unique: the flow on $D_i^L$ is read from
$(b_i^L,t_i)$, the flow on $D_i^R$ from $(s_i,a_i^R)$, and the flow on
$P_i^e$ from $(b_e^L,a_e^R)$. Thus the edge-flow polytope is affinely
isomorphic to the corresponding path-flow simplex.

Fix $s>0$ and $0<\delta<1/(n+1)$, and put
\[
 A=s\bigl(1-(n-1)\delta\bigr),\qquad
 w=s\delta,\qquad d_i=A+s.
\]
At the candidate profile $x^*$, player $i$ sends $A$ on each of
$D_i^L,D_i^R$, sends $w$ on $P_i^e$ for every $e\ne i$, and sends zero on
$P_i^i$. Every congestible arc then has load $A+(n-1)w=s$. The used paths
span a face $\mathcal F$, and $x^*\in\relint(\mathcal F)$.

The extremal direction $q_i$ from
\eqref{eq:fw-counterexample-profile-directions} appears in each resource
layer: increase the flow on every $P_i^e$ with $e\ne i$ by $\theta_i$ and
decrease the flow on each of $D_i^L,D_i^R$ by
$(n-1)\theta_i/2$. This preserves player $i$'s demand and induces the
resource-flow variation $\theta_i q_i$ in both layers. Since all these
paths carry positive flow at $x^*$, the variation lies in
$T_{\mathcal F}$ and is feasible for either sign of sufficiently small
$\theta_i$. The determinant calculation below also accounts for the other
tangent directions of this face.

\begin{lemma}[Equilibrium on the diagonal network]
\label{lem:diagonal-network-equilibrium}
Assume $f'(s)>0$. Let $c=\alpha f+\beta$, where $\alpha>0$ and
\[
 \beta=\alpha\bigl((A-2w)f'(s)-f(s)\bigr),
\]
put cost $c$ on every congestible arc, and initially put zero cost on the
connectors. Then $x^*$ is a Nash equilibrium. The unused path $P_i^i$ has a
strict marginal-cost gap $2(A-w)c'(s)>0$ for every player $i$.
\end{lemma}

\begin{proof}
The marginal costs of the used paths are
\[
 \mu_i(D_i^L)=\mu_i(D_i^R)=c(s)+Ac'(s),\qquad
 \mu_i(P_i^e)=2c(s)+2wc'(s)\quad(e\ne i).
\]
They coincide because $c(s)=(A-2w)c'(s)$. Convexity of each player's
objective in her own flow therefore makes $x^*$ a Nash equilibrium. Moreover,
$\mu_i(P_i^i)=2c(s)+2Ac'(s)$, which exceeds the common used-path value by
$2(A-w)c'(s)$. Here $A>w$ and $c'(s)=\alpha f'(s)>0$.
\end{proof}

\begin{lemma}[Determinant on the diagonal face]
\label{lem:diagonal-network-determinant}
Write $a=c'(s)$ and $b=c''(s)$, and let $T_{\mathcal F}$ be the tangent space
of the face in edge-flow coordinates. If $a,b>0$, then
\begin{equation}
\label{eq:diagonal-network-determinant-sign}
 \operatorname{sign}\det\bigl(DF(x^*)|_{T_{\mathcal F}}\bigr)
 =\operatorname{sign}
 \left(\gamma_na-\bigl(1-(n+1)\delta\bigr)sb\right).
\end{equation}
\end{lemma}

\begin{proof}
The path-to-edge map is injective on $\mathcal F$: every used path variable is
recovered from one of its congestible-arc coordinates. Hence the Gram matrix
of any path-coordinate basis of $T_{\mathcal F}$ is positive definite, and
the sign of the compressed determinant is the sign of the corresponding
bilinear-form matrix before multiplication by the inverse Gram matrix.

Let $u_i,v_i$ denote variations on $D_i^L,D_i^R$, and let $z_{i,e}$ denote
the variation on $P_i^e$, $e\ne i$. The demand constraint is
$u_i+v_i+\sum_{e\ne i}z_{i,e}=0$. Decompose
\[
 h_i=\frac{u_i-v_i}{2},\qquad
 g_i=\frac{u_i+v_i}{2}=-\frac12\sum_{e\ne i}z_{i,e}.
\]
The two layers decouple the antisymmetric variables $h$ from the symmetric
variables $z$. On the $h$-space the bilinear form is
$2(2a+Ab)\langle h,\widehat h\rangle$, and is therefore positive definite.

For the $z$-space, let $X$ be the one-layer player--resource variation matrix,
so that
\[
 X_{i,e}=z_{i,e}\quad(e\ne i),\qquad
 X_{i,i}=-\frac12\sum_{e\ne i}z_{i,e}.
\]
Let $Y$ be another tangent variation of the same form. Write $r_X=X\ones$
and $c_X=X^\top\ones$, and define $r_Y$ and $c_Y$ analogously. We
differentiate the marginal-cost operator in direction $X$ and take the
Euclidean inner product of the resulting variation with $Y$. After
division by the factor two from the identical layers, the induced
bilinear form is
\begin{equation}
\label{eq:diagonal-network-bilinear-form}
 B(X,Y)=a\langle X,Y\rangle
 +(a+wb)c_X^\top c_Y-(A-w)b\,c_X^\top r_Y.
\end{equation}
In the off-diagonal coordinates $z$, the Gram matrix $G$ of the first term is
block diagonal by player, with blocks
$I+\tfrac14\ones\ones^\top$ and inverse
$I-\tfrac1{n+3}\ones\ones^\top$. Define the linear maps
$\mathsf C z:=c_X$ and $\mathsf R z:=r_X$. Direct calculation gives
\begin{align*}
 \mathsf C G^{-1}\mathsf C^\top
 &=\frac{n^2+3n-1}{n+3}I-\frac{n+2}{n+3}\ones\ones^\top,\\
 \mathsf R G^{-1}\mathsf C^\top
 &=-\frac{n+1}{n+3}I+\frac{2}{n+3}\ones\ones^\top.
\end{align*}
The matrix determinant lemma reduces the nonconstant determinant factors to
those of the symmetric $n\times n$ matrix
\[
 aI+(a+wb)\mathsf C G^{-1}\mathsf C^\top
 -(A-w)b\mathsf R G^{-1}\mathsf C^\top.
\]
On $\spanop\{\ones\}$ the factor is
\begin{align*}
 \phi_1
 &=a+\frac{n-1}{n+3}\bigl(a+(2w-A)b\bigr)\\
 &=\frac{n-1}{n+3}
 \left(\gamma_na-\bigl(1-(n+1)\delta\bigr)sb\right),
\end{align*}
whereas on $\ones^\perp$ it is
\[
 \phi_\perp
 =\frac{(n^2+4n+2)a+\bigl((n+1)s+(2n-1)w\bigr)b}{n+3}>0
\]
with multiplicity $n-1$. All remaining factors, including $\det G$, are
positive. Together with the positive $h$-block, this proves
\eqref{eq:diagonal-network-determinant-sign}.
\end{proof}

\begin{proof}[Proof of the uniqueness statement in
Theorem~\ref{thm:directed-network-characterizations}]
Sufficiency follows from Theorem~\ref{thm:uniqueness-characterization}, since
network flow polytopes are compact and convex.

For necessity of $Q_n(\Rp)$, take $f\in\C\setminus Q_n(\Rp)$ and choose
$s>0$ such that $sf''(s)>\gamma_nf'(s)$. This inequality implies
$f''(s)>0$. It also implies $f'(s)>0$: otherwise the nonnegative function
$f'$ would have an interior minimum at $s$, forcing $f''(s)=0$. Choose
$\delta\in(0,1/(n+1))$ small enough that
\[
 \gamma_nf'(s)<\bigl(1-(n+1)\delta\bigr)sf''(s),
\]
and construct the diagonal network above. Put $c=\alpha f+\beta$ on its
congestible arcs, with $\beta$ as in
Lemma~\ref{lem:diagonal-network-equilibrium}. That lemma makes $x^*$ a
strictly complementary equilibrium on $\mathcal F$, while
Lemma~\ref{lem:diagonal-network-determinant} makes its restricted determinant
negative.

It remains to replace the auxiliary zero connector costs by costs from $\C$.
For a connector $a$, let $\ell_a^*$ be its load at $x^*$ and let $\xi_a^*$ be
the common player flow among the players using it. Every connector other than
$(b_e^L,a_e^R)$ is used by a single player, so $\xi_a^*=\ell_a^*$; each
diagonal connector $(b_e^L,a_e^R)$ is used by $n-1$ players, each with flow
$\xi_a^*=w$, and has load $\ell_a^*=(n-1)w$. Give it the cost
\[
 c_a^\varepsilon(t)=\varepsilon f(t)
 -\varepsilon\bigl(f(\ell_a^*)+\xi_a^*f'(\ell_a^*)\bigr).
\]
Its marginal cost at $x^*$ vanishes for every player using it. Thus $x^*$
remains exactly an equilibrium. For sufficiently small $\varepsilon>0$, the
strict unused-path gap and the negative determinant persist. All arc costs
belong to $\operatorname{Aff}_+(f)\subseteq\C$, so
Lemma~\ref{lem:face-index} yields a second Nash equilibrium.

Finally, suppose $f\in\C\setminus\D_{\rm mc-si}(\Rp)$. Since
$x\mapsto f(x)+xf'(x)$ is nondecreasing, it is constant, say equal to
$\gamma$, on some nondegenerate interval $[a,b]$. Put $c=f-\gamma$ and give
one player demand $a+b$ on two parallel directed arcs, both with cost $c$.
For every first-arc flow $y\in[a,b]$, the second-arc flow $a+b-y$ also lies in
$[a,b]$, and both paths have marginal cost zero. Hence there is a continuum
of Nash equilibria. The other $n-1$ players can be assigned single-path
commodities in disjoint components. This proves necessity of
$\D_{\rm mc-si}(\Rp)$ for network games.
\end{proof}

\begin{lemma}[Positive costs by vertex-potential shifts]
\label{lem:positive-network-costs}
Consider a fixed-demand directed network game with resource costs
$c_e\in\D(\Rp)$. For any vertex potential $\psi:V\to\R$, replacing
\[
 c_e(t)\quad\text{by}\quad
 \widetilde c_e(t)=c_e(t)+\psi(v)-\psi(u),\qquad e=(u,v),
\]
preserves all Nash equilibria and the entire Euclidean-regularized FW
response map, for every $\eta>0$. If the graph is acyclic, the potential
can be chosen so that every $\widetilde c_e$ is strictly positive on $\Rp$.
If the original costs belong to a PAT-closed class, so do the shifted costs.
\end{lemma}
\begin{proof}
For every feasible flow $y_i$ of player $i$, flow conservation gives
\[
 \sum_{e=(u,v)}(\psi(v)-\psi(u))y_{i,e}
 =d_i(\psi(t_i)-\psi(s_i)).
\]
Thus the change in her objective is constant on $K_i$. The same constant
is added to the linear objective of the regularized oracle, since the
marginal cost on arc $e$ changes by $\psi(v)-\psi(u)$. Both minimization
problems therefore retain their minimizers.

If the graph is acyclic, take a topological ordering $\tau$ with
$\tau(v)-\tau(u)\ge1$ on every arc and set $\psi=M\tau$. Choose
\[
 M>\max\bigl(\{0\}\cup\{-c_e(0):e\in E\}\bigr).
\]
Nondecreasing costs then satisfy
$\widetilde c_e(t)\ge c_e(0)+M>0$ for all $t\ge0$. Finally, adding a
constant preserves membership in a PAT-closed class.
\end{proof}

All directed necessity constructions above are acyclic, including the
parallel-arc construction for failure of strict increase of marginal costs.
Lemma~\ref{lem:positive-network-costs} therefore proves the positive-cost
assertion of Theorem~\ref{thm:directed-network-characterizations}.

\begin{corollary}[Degree-six counterexamples on directed networks]
\label{cor:directed-degree-six}
There are three-player atomic splittable games on acyclic directed networks
with polynomial arc costs of degree six and nonnegative coefficients such
that, respectively, an isolated regularized equilibrium is locally unstable
under Euclidean-regularized FW dynamics, and Nash equilibria are not unique.
The arc costs can be chosen strictly positive on $\Rp$.
\end{corollary}
\begin{proof}
Take $f(t)=t^6/6$. Since $f'(1)=1$ and $f''(1)=5>\gamma_3=4$, the two
directed constructions apply. Every resulting cost is
$\alpha_e f+\beta_e$ with $\alpha_e>0$. By
Lemma~\ref{lem:positive-network-costs}, vertex-potential shifts preserve
instability and non-uniqueness while making every constant coefficient
strictly positive. The degree-six coefficients are already positive.
\end{proof}

For three players, among games whose arc costs are nonconstant polynomials
with nonnegative coefficients, degree six is minimal for both directed-network
counterexamples in Corollary~\ref{cor:directed-degree-six}: every nonconstant
polynomial with nonnegative coefficients of degree at most five satisfies
$xf''(x)\le4f'(x)$ term by term and belongs to
$Q_3(\Rp)\cap\D_{\rm mc-si}(\Rp)$. The sufficiency directions of
Theorems~\ref{thm:fw-characterization} and~\ref{thm:uniqueness-characterization}
therefore guarantee FW stability and uniqueness for all such games.

\subsection{FW-Stability on Parallel Networks}
\label{sec:parallel-stability}

We prove Theorem~\ref{thm:parallel} by showing that the compressed
marginal-cost Jacobian has real, nonnegative spectrum, even though parallel
links need not restore VI-monotonicity.

\begin{remark}\label{rem:parallel}
The interval construction from the necessity proof of
Theorem~\ref{thm:classification} can itself be represented by two parallel
links if $\C$ contains affine costs. Identify $[0,d_i]$ with
$
 \{(x_{i1},x_{i2})\in\mathbb R_+^2:x_{i1}+x_{i2}=d_i\}$
through $x_i\mapsto(x_i,d_i-x_i)$. Put the critical cost
$f\in\C$ on the first link and an affine cost
$a(t)=\alpha t+\beta\in\C$ on the second. If $u$ is the violating direction
from~\eqref{eq:tight-direction}, the second-link direction is $-u$ and adds
$\alpha\bigl(\norm{u}^2+(\ones^\top u)^2\bigr)$
to the quadratic form. Under PAT-closedness, $\alpha>0$ can be chosen small
enough that the total remains negative. Thus the VI may fail to be monotone
even on a two-link network.
\end{remark}

\begin{proof}[Proof of Theorem~\ref{thm:parallel}]
Let $x^*\in\relint(K)$ be a regularized equilibrium. Put
\[
 U:=\{v\in\R^E:\ones^\top v=0\},\qquad P:=P_U,\qquad T=U^n.
\]
All players have the same link set, so $P_T$ applies $P$ to every player
block. By Lemma~\ref{lem:derivative-beta}, the FW linearization is
\[
 DG(x^*)|_T=-I-\eta^{-1}A,\qquad A:=DF(x^*)|_T.
\]
We prove that the spectrum of $A$ is real and nonnegative.
For each edge define
\[
 \ell_e:=\bar x_e^*,\qquad a_e:=c'_e(\ell_e)\ge0,\qquad
 b_e:=c''_e(\ell_e),
\]
and
\[
 D_e:=(n+1)a_e+b_e\ell_e
     =(n-1)a_e+\Theta_{c_e}(\ell_e)\ge0.
\]
For $h\in T$, let $\bar h=\sum_i h_i$ be the aggregate-load variation.
The resource derivative formula~\eqref{eq:DF-action-prelim} gives
\begin{equation}\label{eq:parallel-DF-action}
 (DF(x^*)h)_{i,e}=a_eh_{i,e}+(a_e+b_ex^*_{i,e})\bar h_e.
\end{equation}
Consider the surjective linear map $S:T\to U$ defined by $Sh=\bar h$.
Summing~\eqref{eq:parallel-DF-action} over players and applying $P$ yields
\[
 SA=A_DS,\qquad A_D:=P\diag(D)|_U.
\]
Consequently $\ker S$ is invariant under $A$, and the operator induced by
$A$ on $T/\ker S$, identified with $U$ through $S$, is $A_D$.
On $\ker S$ the aggregate-load term vanishes, so
$(Ah)_i=A_ah_i$, where $A_a:=P\diag(a)|_U$.
The isomorphism
\[
 U^{n-1}\longrightarrow\ker S,\qquad
 (h_1,\ldots,h_{n-1})\longmapsto
 (h_1,\ldots,h_{n-1},-\textstyle\sum_{i<n}h_i)
\]
identifies this restriction with $n-1$ copies of $A_a$. A basis adapted to
$\ker S$ therefore makes $A$ block triangular, with these copies and $A_D$
on its diagonal. In particular,
\begin{equation}\label{eq:parallel-characteristic-factorization}
 \det(\lambda I_T-A)
 =\det(\lambda I_U-A_a)^{n-1}\det(\lambda I_U-A_D).
\end{equation}
Both $A_a$ and $A_D$ are symmetric positive semidefinite: for $v\in U$,
\[
 \langle v,A_av\rangle=\sum_e a_ev_e^2\ge0,\qquad
 \langle v,A_Dv\rangle=\sum_e D_ev_e^2\ge0.
\]
Thus every eigenvalue $\mu$ of $A$ is real and nonnegative. Every
eigenvalue $-1-\mu/\eta$ of the FW linearization is at most $-1$.
Corollary~\ref{cor:fw-spectral-criterion} proves local exponential
stability. If there is only
one link, $K$ is a singleton and the conclusion is immediate.
\end{proof}

\begin{remark}
Factorization~\eqref{eq:parallel-characteristic-factorization} holds at every
feasible profile on the common parallel-link strategy sets, without the upper
curvature bound or positive cost derivatives. It concerns the spectrum, not
the symmetric part (Remark~\ref{rem:parallel}); interiority is needed only for
the smooth FW linearization.
\end{remark}

\bibliographystyle{abbrv}
\bibliography{arxiv-paper}

\end{document}